\documentclass[12pt,reqno]{amsart} 
\usepackage{setspace}
\usepackage{amsmath}\allowdisplaybreaks
\usepackage{amscd,amssymb,stmaryrd,amstext,amsthm}
\usepackage{url}
\usepackage{subfigure}
\usepackage{enumitem}
\usepackage{graphicx}
\usepackage{color}
\usepackage{epstopdf}
\usepackage{caption}
\usepackage{multicol}
\usepackage{tikz}
\usepackage{forest}
\usepackage{pgfplots} 
\pgfplotsset{compat=newest}
\usetikzlibrary{decorations.pathreplacing,angles,quotes}
\usetikzlibrary{bending}
\usepackage{mathtools}
\usepackage[colorlinks]{hyperref}
\usepackage[all]{hypcap}
\usepackage[lite]{amsrefs}
\usepackage{amsfonts}
\usepackage[normalem]{ulem}

\newcounter{desccount}

\newcommand{\descref}[1]{\hyperref[#1]{#1}}

\DeclareFontFamily{U}{mathx}{}
\DeclareFontShape{U}{mathx}{m}{n}{<-> mathx10}{}
\DeclareSymbolFont{mathx}{U}{mathx}{m}{n}
\DeclareMathAccent{\widecheck}{0}{mathx}{"71}

\numberwithin{equation}{section}

\newcommand{\sub}{\subseteq}

\newcommand{\supp}{\mathrm{supp}}

\newcommand{\Z}{\mathbb{Z}}
\newcommand{\R}{\mathbb{R}}
\newcommand{\C}{\mathbb{C}}

\renewcommand{\P}{\mathcal{P}}

\numberwithin{chap}{section}
\newtheorem{thm}{Theorem}
\numberwithin{thm}{section}

\newtheorem{prop}[thm]{Proposition}
\newtheorem{defn}[thm]{Definition}
\newtheorem{lem}[thm]{Lemma}

\newtheorem{cor}[thm]{Corollary}

\DeclarePairedDelimiter{\norm}{\lVert}{\rVert}

\begin{document}

\pagestyle{myheadings} \thispagestyle{empty} \markright{}
\title{$L^2$ affine Fourier restriction theorems for smooth surfaces in $\R^3$}
\author{Jianhui Li}
\date{}

\maketitle

\begin{abstract}
    We prove sharp $L^2$ Fourier restriction inequalities for compact, smooth surfaces in $\R^3$ equipped with the affine surface measure or a power thereof. The results are valid for all smooth surfaces and the bounds are uniform for all surfaces defined by the graph of polynomials of degrees up to $d$ with bounded coefficients. The primary tool is a decoupling theorem for these surfaces.
\end{abstract}

\section{Introduction}

\subsection{Background and main results}

Let $S$ be a smooth, compact surface in $\R^3$ and $d\Sigma$ be the induced Lebesgue measure on $S$. By standard localization argument, we assume without loss of generality that $S$ is defined by the graph of a smooth function $\phi$ over $[-1,1]^2$. Then we have the following Stein-Tomas restriction theorem.

\begin{thm}\cites{Stein1993,To75}
Let $S$, $\phi$ and $d\Sigma$ be as above. Suppose in addition that $S$ has non-vanishing Gaussian curvature everywhere, then for any $f\in \mathcal{S}(\R^3)$,
\begin{equation}\label{eq:T-S restriction}
    \|\hat {f}\|_{L^2(d\Sigma)} \leq C \|f\|_{L^{4/3}},
\end{equation}
where the constant $C$ depends only on the lower bounds of the principal curvatures of $S$ and the $C^3$ norm of $\phi$. 

\end{thm}

For the case where $S$ has vanishing Gaussian curvature somewhere, \cite{IM2016} provides a necessary and sufficient condition on $p=p(S)<4/3$ for a large class of smooth surfaces that includes all real analytic surfaces such that 
\begin{equation}\label{eq:I-M restriction}
    \|\hat {f}\|_{L^2(d\Sigma)} \lesssim_S \|f\|_{L^{p}},
\end{equation}
holds. See Theorems 1.14 and 1.15 in \cite{IM2016}. See also \cites{IM2011,M2014,Pa21,PS1997,Var76} for previous developments along the line.

In this paper, we employ the affine surface measure on $S$:
\begin{equation}
    d\mu(\xi,\phi(\xi)) = |\det D^2\phi(\xi) |^{\frac{1}{4}}d\xi
\end{equation}
and the extra damped version:
\begin{equation}
    d\mu_\varepsilon(\xi,\phi(\xi)) = |\det D^2\phi(\xi) |^{\frac{1}{4}+\varepsilon}d\xi, \quad \text{for }\varepsilon>0,
\end{equation}
to replace $d\Sigma$ in \eqref{eq:T-S restriction}. The affine surface measure has also been considered in \cites{Sh07,Sh09,CKZ07} for surfaces of revolutions, in \cite{CKZ2013} for surfaces given by graphs of homogeneous polynomials, and in \cite{Pa20} for surfaces given by the graphs of mixed homogeneous polynomials. Other notable results includes \cites{Oberlin2012,Hi14}. 

The affine surface measure is a special case of the affine Hausdorff measure introduced in \cite{Gr19}. While the affine restriction problem on hypersurfaces is still largely unknown in higher dimensions, the similar problem on curves in $\R^d$ has been fully resolved, see for instance \cites{BOS09, BOS13, Stovall2016}.

The main results of this paper are as follows:
\begin{thm}\label{thm:main}
Let $\varepsilon > 0$ and $R\geq 1$. Let $S$, $\phi$, $d\mu$ and $d\mu_\varepsilon$ be as above.  We have the following estimates:

\begin{enumerate}
    \item for any ball $B_R$ of radius $R$ and any $f \in \mathcal{S}(\R^3)$ compactly supported on $B_R$,
    \begin{equation} \label{eq:affine-rest}
    \|\hat {f}\|_{L^2(d\mu)} \lesssim_{\varepsilon,\phi} R^{\varepsilon} \|f\|_{L^{4/3}(B_R)};
    \end{equation}
    \item for any $f \in \mathcal{S}(\R^3)$,
    \begin{equation}\label{eq:affine-rest_damped}
    \|\hat {f}\|_{L^2(d\mu_\varepsilon)} \lesssim_{\varepsilon,\phi} \|f\|_{L^{4/3}}.
    \end{equation}
\end{enumerate}

Moreover, the implicit constants in \eqref{eq:affine-rest} and \eqref{eq:affine-rest_damped} can be made uniform for all polynomials $\phi$ of degrees up to $d$ with bounded coefficients.
\end{thm}

The reader will see in the proof below that the endpoint over-damped result \eqref{eq:affine-rest_damped} implies \eqref{eq:affine-rest} by elementary arguments. Our focus is therefore the endpoint over-damped result.

By abuse of notation, we define $d\mu_{-1/4}$ on $S$ to be the pullback measure of the two dimensional Lebesgue measure $d\xi$. Then by trivial estimates, we have for any $f\in \mathcal{S}(\R^3)$, 
\begin{equation} \label{eq:trivial_L1_L2}
    \|\hat f\|_{L^2(d\mu_{-1/4})} \lesssim \|f\|_{L^1}
\end{equation}
and
\begin{equation} \label{eq:trivial_L1_Linfty}
    \|\hat f\|_{L^\infty(d\mu)} \lesssim \|f\|_{L^1}.
\end{equation}

By complex interpolation, see for instance \cite{SW58}, among \eqref{eq:affine-rest_damped}, \eqref{eq:trivial_L1_L2} and \eqref{eq:trivial_L1_Linfty}, we get the following affine restriction estimates off the scaling line for affine surface measure $\mu$.

\begin{cor} \label{cor:main}
Let $2\leq q < \infty$ and $2 - \frac{2}{p}<\frac{1}{q}$. Let $S$, $\phi$ and $d\mu$ be as above. For any $f \in \mathcal{S}(\R^3)$, we have
\begin{equation}\label{eq:cor}
    \|\hat f\|_{L^q(d\mu)} \lesssim_\phi \|f\|_{L^p}.
\end{equation}

Moreover, the implicit constant in \eqref{eq:cor} can be made uniform for all polynomials $\phi$ of degrees up to $d$ with bounded coefficients.
\end{cor}

We remark that Theorem \ref{thm:main} and Corollary \ref{cor:main} are sharp for the case of general smooth surfaces by considering the following counterexample modified from Sj\"olin's two-dimensional example in \cites{Sj74}.

\begin{prop}\label{prop:counterexample}
Let $1\leq q <\infty$ and  $2-\frac{2}{p} = \frac{1}{q}$. Let $\phi(\xi) = e^{-1/|\xi|}\sin(|\xi|^{-3})$ if $\xi\neq 0$, and $\phi(0)=0$. Let $d\mu$ be as above.  Then there is no constant $C$ such that the following holds for all $f\in \mathcal{S}(\R^3)$:
\begin{equation}\label{eq:counterexample}
    \|\hat f\|_{L^q(d\mu)} \leq C \|f\|_{L^p}.
\end{equation}
\end{prop}

See Section \ref{sec:counterexample} for the proof of Proposition \ref{prop:counterexample}. See also the examples in Theorem 1.3 of \cite{CZ02}. Note that $\phi$ is highly oscillating in both examples. It is still unknown whether the $\varepsilon$ loss is necessary for convex surfaces of finite type or polynomial surfaces.

Instead of directly proving Theorem \ref{thm:main}, we prove the following equivalent version. See Proposition 1.27 of \cite{Demeter2020}. In fact, we will prove the superficially stronger estimates with $B_R$ replaced by $\R^3$ on the right hand sides of \eqref{eq:affine-rest_eqv} and \eqref{eq:affine-rest_damped_eqv}. For completeness, we include the proof of Theorem \ref{thm:main} by using Theorem \ref{thm:main_eqv} in Appendix.

\begin{thm}\label{thm:main_eqv}
Let $\varepsilon > 0$ and $R\geq 1$. Let $S$ and $\phi$ be as above. Define measures $M,M_\varepsilon$ on $(\xi,\eta) \in [-1,1]^2 \times \R$ by
\begin{equation}
    dM(\xi,\eta) =dM^\phi(\xi,\eta) = |\det D^2 \phi(\xi)|^{-1/4} d\xi d\eta
\end{equation}
and
\begin{equation}
    dM_\varepsilon(\xi,\eta) = dM_\varepsilon^\phi(\xi,\eta)= |\det D^2 \phi(\xi)|^{-1/4-\varepsilon} d\xi d\eta.
\end{equation}
Then for any ball, $B_R$ of radius $R$ and any function $F$ such that $\hat F$ is supported on the $R^{-1}$ vertical neighborhood of the graph of $\phi$ above $[-1,1]^2$,  we have
\begin{equation}\label{eq:affine-rest_eqv}
    \|F\|_{L^4(B_R)} \lesssim_{\varepsilon,\phi} R^{\varepsilon-1/2} \|\hat F\|_{L^2(dM)} \quad \text{if }\hat F \in L^2(dM)
\end{equation}
and
\begin{equation}\label{eq:affine-rest_damped_eqv}
    \|F\|_{L^4(B_R)} \lesssim_{\varepsilon,\phi} R^{-1/2} \|\hat F\|_{L^2(dM_\varepsilon)}\quad \text{if }\hat F \in L^2(dM_\varepsilon).
\end{equation}

Moreover, the implicit constants in \eqref{eq:affine-rest_eqv} and \eqref{eq:affine-rest_damped_eqv} can be made uniform for all polynomials $\phi$ of degrees up to $d$ with bounded coefficients.
\end{thm}

\subsection{Notation} 
\begin{enumerate}
    \item We use $(\xi,\eta) \in \R^2 \times \R$ to denote the frequency variables. The surface $S$ is given by $\eta = \phi(\xi)$ for $\xi \in [-1,1]^2$.
    \item $\mathcal{S}(\R^3)$, or sometimes $\mathcal{S}$, denotes the space of all Schwartz functions in $\R^3$.
    \item We use the standard notation $a=O_A(b)$, or $a\lesssim_A b$ to mean there is a constant $C$ depending on $A$ that may change from line to line such that $a\leq Cb$. In many cases, such as results related to polynomials, the constant $C$ will be taken to be absolute or depend on the polynomial degree only. So we may simply write $a=O(b)$ or $a\lesssim b$. We also use $a \sim_A b$ to mean $a \lesssim_A b$ and $b \lesssim_A a$.
    \item Dyadic numbers are numbers of the form $(1+c)^{-k}, k\in \Z_{\ge 0}$, for some positive constant $c$ that will not change throughout the paper.
    \item For any parallelogram $\Omega\subset \R^2$, the width of $\Omega$ is defined to be the shorter of the two distances between opposite sides.
    \item For any set $A\subseteq \R^2$ (or $A\subseteq\R$), $1_{A}$ denotes the characteristic function on $A$.
    \item For any set $A\subseteq \R^2$ (or $A\subseteq\R$), any continuous function $\phi:A\to \R$ and any $\delta>0$, denote the (vertical) $\delta$-neighbourhood of the graph of $\phi$ above $A$ by
    $$
    \mathcal N^\phi_{\delta}(A)=\{(\xi,\eta):\xi\in A, |\eta-\phi(\xi)|<\delta\},
    $$
    with the obvious modification for $A\subseteq \R$.
    \item For $F:\R^3\to \C$ and $A\sub \R^2$, we denote by $F_A$ the Fourier restriction of $F$ to the strip $A\times \R$, namely, $F_A$ is defined by the relation
        \begin{equation}\label{eqn_Fourier_restriction_strip}
            \widehat {F_A}(\xi,\eta)=\hat F(\xi,\eta)1_A(\xi).
        \end{equation}
    \item For any smooth function $\phi : \R^2 \to \R$, we define the norm of $\phi$ to be
    \begin{equation}
        \|\phi\| : = \sup_{[-1,1]^2}|\phi|.
    \end{equation}
    \item  \label{item:poly} For any polynomial $P:\R^2 \to \R$ defined by $P=\sum_{\alpha} c_\alpha \xi^\alpha$ of degree at most $d$, we have the following equivalence:
    $$
    \|P\|:=  \sup_{[-1,1]^2}  |P| \sim_d \max_{\alpha} |c_{\alpha}| \sim_d \sum_{\alpha} |c_\alpha| .
    $$
    We say a polynomial $P$ is bounded if $\|P\|\lesssim_d 1$.
\end{enumerate}

{\it Outline of the paper.} In Section \ref{sec:main_idea}, we present the main ideas of the proof of Theorem \ref{thm:main_eqv}. The rescaling proposition, Proposition \ref{prop:est_on_each_piece} and the key decoupling inequalities, Theorem \ref{thm:dec_main} and Corollary \ref{cor:dec_smooth}, are introduced with remarks. In Section \ref{sec:main_result_proof}, we prove Theorem \ref{thm:main_eqv} by assuming these results. In Section \ref{sec_rescaling},\ref{sec:proof_dec} and \ref{sec:cor:dec_smooth}, we prove Proposition \ref{prop:est_on_each_piece}, Theorem \ref{thm:dec_main} and Corollary \ref{cor:dec_smooth} respectively. In Section \ref{sec:counterexample}, we prove Proposition \ref{prop:counterexample}. In Appendix, we prove that Theorem \ref{thm:main_eqv} implies \ref{thm:main}.

{\it Acknowledgement.} The author would like to thank his advisor Betsy Stovall for her advice and help throughout the project. He would like to thank Tongou Yang for his valuable discussions. He would also like to thank the anonymous referee for their careful reading of the manuscript and their insightful suggestions. The author was partially supported by NSF DMS-1653265 and the Wisconsin Alumni Research Foundation.

\section{Main ideas of the proof} \label{sec:main_idea}


The proof of Theorem \ref{thm:main_eqv} relies on a variant of the decoupling inequalities proved by Yang and the author in \cite{LY2021.2}. While one can directly apply the $\ell^2(L^4)$ decoupling inequality stated in Theorems 1.3 and 1.4 of \cite{LY2021.2} to prove \eqref{eq:affine-rest_eqv} for convex cases, the case where the principal curvatures have different signs needs necessary modifications. Instead of decoupling the surface into flat pieces, we stop at some intermediate stages at which each decoupled piece can be rescaled to a surface with the Hessian determinant bounded away from zero. The affine surface measure respects any affine transformation, and hence we can then apply Theorem \ref{eq:T-S restriction} to get an estimate on each piece. The $\ell^2$ decoupling inequality allows us to sum over all pieces to give the desired result.

We need some preparation before introducing the key decoupling inequality. 

Let $\Omega \subseteq [-1,1]^2$ be a parallelogram. Define $T_{\Omega}$ be the invertible affine map that maps $[-1,1]^2$ to $\Omega$. Consider the following function:
\begin{equation} \label{eq:phi_T_Omega}
    \phi_{T_\Omega}(\xi):=\phi \circ T_\Omega(\xi) - \nabla (\phi \circ T_\Omega)(0,0) \cdot \xi - \phi \circ T_\Omega(0,0).
\end{equation}
and its normalisation
\begin{equation}\label{eq:bar_phi_T_Omega}
    \bar{\phi}_{T_\Omega} = \frac{\phi_{T_\Omega}}{\|\phi_{T_\Omega}\|},
\end{equation}
where 
$
\|\phi_{T_\Omega}\| := \sup_{[-1,1]^2} |\phi_{T_\Omega}|.
$

We have the following definition of admissible sets:
\begin{defn}\label{def:admissible}
For $\varepsilon>0$ and $0<\sigma\leq 1$, a parallelogram $\Omega\subseteq [-1,1]^2$ is said to be $(\phi,\sigma,R,\varepsilon)$-admissible if either of the following holds:
\begin{enumerate}
    \item\label{item:1:def_flat} $\sigma^{21} \leq R^{-1}$,  $|\det D^2 \phi| \lesssim \sigma$ on $10\Omega$, and $\sigma |\Omega|^2 \lesssim_\varepsilon R^{-2}$. Furthermore, $\|\phi_{T_{\Omega}}\|\lesssim R^{-1}$  if $\phi$ is a bounded polynomial of degree $d$.
    \item\label{item:2:def_curve}  $|\det D^2 \phi| \sim \sigma$ over $10\Omega$, the width of $\Omega$ is $\gtrsim \sigma^{21}$ and for any $\xi\in [-1,1]^2,$
    \begin{equation}\label{eq:def:admissible}
        |\det D^2 \bar{\phi}_{T_\Omega}(\xi) | \sim_\varepsilon 1 \quad \text{and} \quad \sum_{|\alpha|=2,3} |D^{\alpha}\bar{\phi}_{T_\Omega}(\xi)|\lesssim_\varepsilon 1.
    \end{equation}
\end{enumerate}
\end{defn}


To facilitate the understanding of Definition \ref{def:admissible}, several remarks are made:
\begin{enumerate}
    \item $\bar{\phi}_{T_\Omega}$ over $[-1,1]^2$ is a normalisation of $\phi$ over $\Omega$. Studying the affine restriction and decoupling problems of $\phi$ on $\Omega$ can be reduced to $\bar{\phi}_{T_\Omega}$ on $[-1,1]^2$.
    \item $\Omega$ satisfying property (\ref{item:1:def_flat}) in the definition is also called $(\phi,R^{-1})$-flat in \cites{LY22,LY2021.2}. Geometrically, $\mathcal{N}_{R^{-1}}^\phi(\Omega)$ is essentially a rectangular box. As a result, no curvature condition can be deployed. 
    \item Note that $\sigma |\Omega|^2 \lesssim_d R^{-2}$ is implied by $\sup_{[-1,1]^2}|\det D^2 \phi_{T_\Omega}| \sim \sigma|\Omega|^2$ and $\|\phi_{T_\Omega}\| \lesssim R^{-1}$ when $\phi$ is a polynomial.
    \item The power $21$ is not optimized and mainly for technical reasons.
    \item $\bar{\phi}_{T_\Omega}$ satisfying \eqref{eq:def:admissible} is well-understood in the $L^2$ Fourier restriction and decoupling problems.
\end{enumerate}
 
The motivation for Definition \ref{def:admissible} is the following proposition:
\begin{prop}\label{prop:est_on_each_piece}
Let $\varepsilon,\varepsilon' > 0$, $R\geq 1$, $0<\sigma\leq 1$ and parallelogram $\Omega \subseteq [-1,1]^2$ be $(\phi,\sigma',R,\varepsilon')$-admissible for some $\sigma'\gtrsim \sigma$. For any function $F$ such that $\supp \hat F \subseteq \mathcal{N}^\phi_{R^{-1}}(\Omega \cap \{ |\det D^2 \phi| \sim \sigma\}),$ we have
\begin{equation}\label{eq:est_on_each_piece}
    \|F\|_{L^4(\R^3)} \lesssim_{\varepsilon,\varepsilon'} R^{-1/2} \sigma^{\varepsilon/2}\|\hat{F}\|_{L^2(dM_\varepsilon)},
\end{equation}
where the implicit constant is independent of $\Omega$. Moreover, the implicit constant can be made uniform for all polynomials $\phi$ of degrees up to $d$ with bounded coefficients.
\end{prop}

This allows us to estimate left hand side of \eqref{eq:affine-rest_damped_eqv} if $\hat{F}$ is further restricted to an admissible set $\Omega$. The proof of Proposition \ref{prop:est_on_each_piece} is given in Section \ref{sec_rescaling}.

Now, we turn to the uniform decoupling theorem for polynomials below:

\begin{thm}[Uniform decoupling theorem for polynomials]\label{thm:dec_main}
Let $\varepsilon > 0$, $2\leq p\leq 6$, $R \gg 1$ and $\phi:\R^2 \to \R$ be a polynomial of degree at most $d$ with bounded coefficients. There exist families $\mathcal{P}_\sigma = \mathcal{P}_\sigma(R,\phi,\varepsilon)$, for dyadic numbers $\sigma\in [R^{-2},1]$, of parallelograms of widths at least $R^{-1}$ such that the following statements hold:
\begin{enumerate}
    \item\label{item1_thm:dec_main} $\mathcal{P}:= \cup_{\sigma} \mathcal P_{\sigma}$, where the union is over dyadic numbers $\sigma \in [R^{-2},1]$, covers $[-1,1]^2$ in the sense that
    \begin{equation}
        \sum_{\Omega \in \mathcal P} 1_{\Omega}(\xi) \geq 1 \quad \text{for }\xi \in [-1,1]^2;
    \end{equation}
    \item $50\mathcal{P}_\sigma$ has $O_\varepsilon(\sigma^{-\varepsilon})$-bounded overlaps in the sense that 
    \begin{equation}\label{eq:ep_bdd_overlap}
        \sum_{\Omega \in \mathcal P_\sigma} 1_{50\Omega} \lesssim_{\varepsilon,d} \sigma^{-\varepsilon};
    \end{equation}
    \item for $\sigma>R^{-2}$, each $\Omega \in \mathcal{P}_\sigma$ is $(\phi,\sigma,R,\varepsilon)$-admissible;
    \item\label{item4_thm:dec_main} for $\sigma=R^{-2}$, $|\det D^2\phi|\lesssim R^{-2}$ over each $\Omega \in \mathcal{P}_{R^{-2}}$.
    \item the following $\ell^2(L^p)$ decoupling inequality holds: for any $F$ whose Fourier transform is supported on $\mathcal{N}_{R^{-1}}^\phi(\cup_{\Omega \in \mathcal{P}_\sigma}\Omega)$,
        \begin{equation}\label{eq:dec_main}
        \|F\|_{L^p(\R^3)} \lesssim_{\varepsilon,d} \sigma^{-\varepsilon} \left ( \sum_{\Omega \in \mathcal{P}_\sigma} \|F_{\Omega}\|_{L^p(\R^3)}^2 \right)^{1/2},
        \end{equation}
    where $F_\Omega$ is defined in \eqref{eqn_Fourier_restriction_strip}.
\end{enumerate}
\end{thm}

By approximating smooth function $\phi$ on each square of size $\sigma^{\varepsilon}$, we obtain the following slightly weaker version of the decoupling theorem as a corollary.

\begin{cor}[Decoupling theorem for smooth functions]\label{cor:dec_smooth}
Let $\varepsilon > 0$, $2\leq p\leq 6$, $R\geq 1$, $R^{-2}< \sigma\leq 1$ be dyadic numbers, and $\phi:\R^2 \to \R$ be a smooth function. Tile $[-1,1]^2$ by squares $Q$ of size $\sigma^\varepsilon$. There exist families $\mathcal{P}^Q_{\ge \sigma} = \mathcal{P}^Q_{\ge \sigma}(R,\phi,\varepsilon)$ of parallelograms of widths at least $R^{-1}$ such that the following statements hold:
\begin{enumerate}
    \item\label{item:1:cor:dec_smooth} $\mathcal{P}^Q_{\ge \sigma}$ covers $\{|\det D^2 \phi|\sim \sigma\}\cap Q$;
    \item\label{item:2:cor:dec_smooth} $50\mathcal{P}^Q_{\ge \sigma}$ has $O_\varepsilon(\sigma^{-\varepsilon})$-bounded overlaps in the sense that 
    \begin{equation}
        \sum_{\Omega \in \mathcal{P}^Q_{\ge \sigma}} 1_{50\Omega} \lesssim_{\varepsilon,\phi} \sigma^{-\varepsilon},
    \end{equation}
    where the implicit constant is independent of $Q$;
    \item\label{item:3:cor:dec_smooth} each $\Omega \in \mathcal{P}^Q_{\ge \sigma}$ is $(\phi,\sigma',R,\varepsilon)$-admissible for some $\sigma'\geq \sigma/2$;
    \item\label{item:4:cor:dec_smooth} the following $\ell^2(L^p)$ decoupling inequality holds: for any $F$ whose Fourier transform is supported on $\mathcal{N}_{R^{-1}}^\phi(\{|\det D^2 \phi|\sim \sigma\}\cap Q)$,
        \begin{equation}\label{eq:cor:dec_smooth}
        \|F\|_{L^p(\R^3)} \lesssim_{\varepsilon,\phi} \sigma^{-\varepsilon} \left ( \sum_{\Omega \in \mathcal{P}^Q_{\ge \sigma}} \|F_{\Omega}\|_{L^p(\R^3)}^2 \right)^{1/2},
        \end{equation}
    where the implicit constant in \eqref{eq:cor:dec_smooth} is independent of $Q$.
\end{enumerate}
\end{cor}

The decoupling inequalities originated from Wolff \cite{Wo2000} as a powerful tool to study local smoothing type estimates. In 2015, Bourgain and Demeter \cites{BD2015,BD2017_study_guide} proved the $\ell^2(L^p)$ decoupling inequalities for hypersurfaces with positive principle curvatures. Later, they \cite{BD2016} also proved the sharp decoupling results for the case where principle curvatures can be negative but not zero. For general cases, decoupling inequalities are shown in \cites{Yang2,BGLSX,Demeter2020} for curves in $\R^2$ with no restriction to curvature. For surfaces in $\R^3$, partial results include \cites{Ke22,Kemp2,LY22,BDK20} while the decoupling inequalities for all smooth surfaces have been proved in \cite{LY2021.2}. 

Let us compare our decoupling results, Theorem \ref{thm:dec_main} and Corollary \ref{cor:dec_smooth}, with previous decoupling results on surfaces with possibly zero Gaussian curvature. First, the goal of all previous decoupling results is to decouple the surface into flat pieces. Unavoidably, these surfaces may contain a line, on which the $\ell^2(L^p)$ ($p>2$) decoupling has loss bigger than $R$ to some positive power. Thus, all previous results consider $\ell^4(L^4)$ decoupling for the case of opposite principal curvatures. However, for the purpose of affine restriction estimates, only losses that are $O_\varepsilon(R^\varepsilon)$, for all $\varepsilon>0$, are tolerable. On the other hand, Theorem \ref{thm:dec_main} and Corollary \ref{cor:dec_smooth} are tailor made to the proof of affine restriction estimates. We decouple the surfaces into admissible pieces, instead of flat pieces. Note that $(\phi,R^{-1})$-flat pieces are automatically admissible. We lessen the requirements each decoupled piece $\Omega$ has to obey, so that $\ell^2(L^p)$ decoupling with $O_\varepsilon(R^\varepsilon)$ loss is possible.

Second, the decoupling constant $O_{\varepsilon,\phi}(\sigma^{-\varepsilon})$ is much smaller than the typical loss $R^{\varepsilon}$ for most $\sigma \in [R^{-1},1]$. For each $\Omega \in \P_{\geq \sigma}^Q$, the extra damped measure $d\sigma_{\varepsilon}$ generates a factor of size $\sigma^{O(\varepsilon)}$ in \eqref{eq:est_on_each_piece}, which can overcome the decoupling loss $O_{\varepsilon,\phi}(\sigma^{\varepsilon})$. We see that these sharper decoupling inequalities are essential to recover the endpoint estimate in the extra damped estimates \eqref{eq:affine-rest_damped} and \eqref{eq:affine-rest_damped_eqv}. 

Third, decoupling results prior to \cite{LY2021.2} are for specific types of surfaces. The partition $\mathcal P$ is explicitly constructed in the proof of these results. On the contrary, the partition in \cite{LY2021.2} is constructed by induction, and depends on $\varepsilon$. This is particularly hard to chase. In this paper, we adopt a different approach than in \cite{LY2021.2} to keep track of the size of each $\Omega$, and the lower bound of $|\det D^2 \phi|$ over $\Omega$. In what follows, we explain these quantities.

We require each $\Omega$ has width at least $R^{-1}$, to guarantee the following condition:
\begin{equation}\label{eq:condition}
    \Omega+ B_{R^{-1}} \approx \Omega.
\end{equation}
This condition turns out to be unnecessary in the proof of the affine restriction theorem, Theorem \ref{thm:main_eqv}. Nevertheless, condition \eqref{eq:condition} is crucial in proving the following formulation of the decoupling inequalities for extension operators, for possibly independent interests.

\begin{cor}
    Let $2\leq p \leq 6$, $R\gg 1$ and $\phi:\R^2 \to \R$ be a polynomial of degree at most $d$ with bounded coefficients. Let $d\mu$ be a Borel measure on $S :=\{(\xi,\phi(\xi)):\xi\in[-1,1]^2\}$ and $d\mu_\Omega$ be restriction of $d\mu$ onto the cap $\{(\xi,\phi(\xi)):\xi\in\Omega\}$, for $\Omega \subseteq [-1,1]^2$. There exists families $\mathcal P_\sigma$ of parallelograms satisfying (\ref{item1_thm:dec_main}) to (\ref{item4_thm:dec_main}) of Theorem \ref{thm:dec_main} such that
        \begin{equation}
            \|\widehat{f d\mu_{\cup \P_\sigma}}\|_{L^p(B)} \lesssim_{\varepsilon,d} \sigma^{-\varepsilon} \left ( \sum_{\Omega \in \mathcal P_\sigma} \|\widehat{f d\mu_\Omega}\|_{L^p(w_B)}^2 \right)^{1/2}
        \end{equation}
    for any ball $B$ of radius $R$ centered at $c(B)$ and $f : S \to \mathbb C $, where $\cup \P_\sigma := \cup_{\Omega \in \P_\sigma}\Omega$ and the weight $w_B$ on $\R^3$ is defined to be
    $$
    w_B(x) : = \left ( 1+ \frac{|x-c(B)|}{R}\right)^{-100}.
    $$
\end{cor}

See Proposition 9.15 of \cite{Demeter2020} for equivalence of the weighted and global formulation.

We will prove Theorem \ref{thm:dec_main} and Corollary \ref{cor:dec_smooth} in Section \ref{sec:proof_dec} and Section \ref{sec:cor:dec_smooth} respectively.

\section{Proof of the main result}\label{sec:main_result_proof} In this section, we prove Theorem \ref{thm:main_eqv} by assuming Proposition \ref{prop:est_on_each_piece}, Theorem \ref{thm:dec_main} and Corollary \ref{cor:dec_smooth}. 

\subsection{The set with tiny Gaussian curvature}\label{subsec:tiny_cur}
In this subsection, we additionally assume $F$ is Fourier supported on $\mathcal{N}_{R^{-1}}^\phi(\{\xi\in [-1,1]^2: |\det D^2 \phi(\xi)| \lesssim R^{-2} \})$. By Hausdorff-Young and H\"older's inequalities, we have
\begin{equation}\label{eq:tiny_cur_1}
    \|F\|_{L^4(\R^3)} \lesssim \|\hat F\|_{L^{4/3}(\R^3)} \lesssim_{\phi} (\text{or} \lesssim_d) (R^{-1})^{3/4-1/2}\| \hat F\|_{L^2(d\xi d\eta)},
\end{equation}
since the support of $\hat{F}$ is of size $O_\phi(R^{-1})$ (or $O_d(R^{-1})$ in the case where $\phi$ is a polynomial). Note that
$$
dM_\varepsilon \gtrsim (R^{-2})^{-1/4-\varepsilon} d\xi d\eta.
$$
Therefore, \eqref{eq:tiny_cur_1} can be further bounded by
\begin{equation}
    \lesssim R^{-1/4} (R^{-1/2-2\varepsilon})^{1/2} \|\hat F\|_{L^2(dM_\varepsilon)} = R^{-1/2-\varepsilon} \|\hat F\|_{L^2(dM_\varepsilon)} \leq R^{-1/2} \|\hat F\|_{L^2(dM_\varepsilon)}.
\end{equation}

This completes the proof of the extra damped estimate \eqref{eq:affine-rest_damped_eqv} with the additional assumption that $\hat F$ is supported on the set with tiny Gaussian curvature.

\subsection{Proof of the extra damped estimate}
In this subsection, we prove the extra damped estimate \eqref{eq:affine-rest_damped_eqv}. First, we decompose \begin{equation}\label{eq:dyadic_dec}
F = F_0+\sum_{\substack{\sigma \text{ dyadic}\\R^{-2}<\sigma \leq 1}} F_\sigma
\end{equation}
where $F_0,F_\sigma \in \mathcal{S}$ such that $F_0$ is Fourier supported on $\mathcal{N}_{R^{-1}}^\phi(\{\xi\in [-1,1]^2: |\det D^2 \phi(\xi)| \lesssim R^{-2} \})$ and $F_\sigma$ is Fourier supported on $\mathcal{N}_{R^{-1}}^\phi(\{\xi\in [-1,1]^2:  |\det D^2 \phi(\xi)| \sim \sigma\})$. 

The required estimate for $F_0$ is done in Section \ref{subsec:tiny_cur}. For $F_\sigma$, we apply Corollary \ref{cor:dec_smooth} with $\varepsilon$ replaced by $\varepsilon/8$ to get $\mathcal P_{\geq \sigma}^Q$ on for each square $Q$ of size $\sigma^{\varepsilon/8}$ in the tiling of $[-1,1]^2$,
\begin{equation}\label{eq:3.3}
    \|(F_\sigma)_Q\|_{L^4(\R^3)}\lesssim_{\varepsilon,\phi} \sigma^{-\varepsilon/8} \left ( \sum_{\Omega \in \mathcal{P}^Q_{\ge \sigma}} \|(F_{\sigma})_{Q \cap \Omega}\|_{L^4(\R^3)}^2 \right)^{1/2}.
\end{equation}

By Proposition \ref{prop:est_on_each_piece} and that $(F_{\sigma})_{Q \cap \Omega}$ is Fourier supported on a $(\phi,\sigma',R,\varepsilon/8)$-admissible set for some $\sigma'\ge \sigma$, we have 
\begin{equation}\label{eq:3.4}
    \|(F_{\sigma})_{Q \cap \Omega}\|_{L^4(\R^3)} \lesssim_\varepsilon R^{-1/2} \sigma^{\varepsilon/2} \|\widehat{ (F_{\sigma})_{Q \cap \Omega}}\|_{L^2(dM_\varepsilon)}.
\end{equation}
Putting \eqref{eq:3.4} to \eqref{eq:3.3} and use the fact that $50\mathcal{P}_{\ge \sigma} \cap Q$ has $O_\varepsilon(\sigma^{-\varepsilon/8})$ overlapping, we obtain
\begin{align*}
    \|(F_\sigma)_Q\|_{L^4(\R^3)} &\lesssim_{\varepsilon,\phi} \sigma^{-\varepsilon/8} \left (\sum_{\Omega \in \mathcal{P}_{\ge \sigma}^Q} R^{-1} \sigma^{\varepsilon} \|\widehat{ (F_{\sigma})_{Q \cap \Omega}}\|_{L^2(dM_\varepsilon)}^2 \right)^{1/2}\\
    &\lesssim_{\varepsilon,\phi} R^{-1/2} \sigma^{\varepsilon/4} \|\widehat{ (F_{\sigma})_Q}\|_{L^2(dM_\varepsilon)},
\end{align*}
for some implicit constant independent of $Q$.

Since there are $\sigma^{-\varepsilon/4}$ many $Q$, by triangle and H\"older's inequalities, we have
\begin{align*}
    \|F_\sigma\|_{L^4(\R^3)} \lesssim (\sigma^{-\varepsilon/4})^{1-1/2}\left ( \sum_Q \|(F_\sigma)_Q\|_{L^4(\R^3)}^2\right)^{1/2} \lesssim_{\varepsilon,\phi} R^{-1/2} \sigma^{\varepsilon/8} \|\hat F_{\sigma}\|_{L^2(dM_\varepsilon)}.
\end{align*}

Finally, we sum up the dyadic pieces:
\begin{align*}
    \|F\|_{L^4(\R^3)} &\leq \|F_0\|_{L^4(\R^3)}+\sum_{\substack{\sigma \text{ dyadic}\\R^{-2}<\sigma \leq 1}} \|F_\sigma\|_{L^4(\R^3)} \\
    &\lesssim_{\varepsilon,\phi} R^{-1/2-\varepsilon}\|\hat F_0\|_{L^2(dM_\varepsilon)} + \sum_{\substack{\sigma \text{ dyadic}\\R^{-2}<\sigma \leq 1}}R^{-1/2} \sigma^{\varepsilon/8} \|\hat F_{\sigma}\|_{L^2(dM_\varepsilon)} \\
    &\leq R^{-1/2}\left (R^{-\varepsilon}+\sum_{\substack{\sigma \text{ dyadic}\\R^{-2}<\sigma \leq 1}} \sigma^{\varepsilon/8} \right )^{1/2}\left (\|\hat F_0\|_{L^2(dM_\varepsilon)}^2 +\sum_{\substack{\sigma \text{ dyadic}\\R^{-2}<\sigma \leq 1}} \|\hat F_\sigma\|_{L^2(dM_\varepsilon)}^2 \right )^{1/2} \\
    &\lesssim R^{-1/2} \|\hat F\|_{L^2(dM_\varepsilon)},
\end{align*}
as desired.

For the case where $\phi$ is a bounded polynomial of degree at most $d$, the same proof applies except every $\lesssim_{\varepsilon,\phi}$ is replaced by $\lesssim_{\varepsilon,d}$. This is because the uniform estimates for the piece $F_0$ and the decoupling inequality are available.

\subsection{Proof of the estimate with the affine surface measure} In this subsection, we prove the estimate with the affine surface measure \eqref{eq:affine-rest_eqv}. In fact, this is implied by \eqref{eq:affine-rest_damped_eqv} and the estimates in Section \ref{subsec:tiny_cur}.

We decompose $F = F_0+F_1$ such that $F_0$ is Fourier supported on $\mathcal{N}_{R^{-1}}^\phi(\{\xi\in [-1,1]^2: |\det D^2 \phi(\xi)| \lesssim R^{-2} \})$ and $F_1$ is Fourier supported on $\mathcal{N}_{R^{-1}}^\phi(\{\xi\in [-1,1]^2:  |\det D^2 \phi(\xi)| \geq R^{-2}\})$. The estimate for $F_0$ is already done in Section \ref{subsec:tiny_cur}. It suffices to estimate $F_1$. 

Note that on the set where $|\det D^2 \phi(\xi)| \geq R^{-2}$, we have
$$
dM_\varepsilon \leq R^{2\varepsilon} dM.
$$
Therefore, by using \eqref{eq:affine-rest_damped_eqv}, we have
\begin{equation}\label{eq_3.6}
    \|F\|_{L^4(B_R)} \lesssim_{\varepsilon,\phi} R^{-1/2} \|\hat F\|_{L^2 (dM_\varepsilon)} \leq R^{2\varepsilon-1/2} \|\hat F\|_{L^2 (dM)}.
\end{equation}
Since \eqref{eq_3.6} is true for arbitrary $\varepsilon>0$, we have obtained \eqref{eq:affine-rest_eqv}.

The case where $\phi$ is a bounded polynomial of degree at most $d$ is similar. We finished proving Theorem \ref{thm:main_eqv} by assuming Proposition \ref{prop:est_on_each_piece}, Theorem \ref{thm:dec_main} and Corollary \ref{cor:dec_smooth}. 

\section{A scaling argument}\label{sec_rescaling} In this section, we prove Proposition \ref{prop:est_on_each_piece} by a scaling argument. 

Since on the support of $\hat F$, we have $|\det D^2 \phi|\sim \sigma$, and hence $dM_\varepsilon \sim \sigma^{-\varepsilon} dM$. It suffices to show that
\begin{equation}\label{eq:4.1}
    \|F\|_{L^4(\R^3)} \lesssim_{\varepsilon'} R^{-1/2} \|\hat{F}\|_{L^2(dM)}.
\end{equation}

We need the following scaling lemma:

\begin{lem}[Affine invariance of measure $M$]
Let $R^{-1}\leq s\leq 1$, $\Omega \subseteq [-1,1]^2$ be a parallelogram and $\phi$ be a smooth function over $[-1,1]^2$. Let $ T_\Omega, \phi_{T_\Omega}$ be defined in \eqref{eq:phi_T_Omega}. Let $\bar{\phi} = s^{-1}\phi_{T_\Omega}$. Let $L:\R^3 \to \R^3$ be an affine transformation that maps $\mathcal{N}_{(sR)^{-1}}^{\bar{\phi}}([-1,1]^2)$ to $\mathcal{N}_{R^{-1}}^\phi(\Omega)$. Let $F$ be such that $\supp \hat F \subseteq \mathcal{N}_{R^{-1}}^\phi(\Omega)$. Define $\hat G (\xi,\eta) = \hat F (L(\xi,\eta))$ such that $\hat G$ is supported on $\mathcal{N}_{(sR)^{-1}}^{\bar{\phi}}([-1,1]^2)$. Then we have
\begin{equation}\label{eq:affine_inv}
    \frac{\|F\|_{L^4(\R^3)}}{R^{-1/2}\|\hat F\|_{L^2(dM^\phi)}} = \frac{\|G\|_{L^4(\R^3)}}{(sR)^{-1/2}\|\hat G\|_{L^2(dM^{\bar{\phi}})}}.
\end{equation}
\end{lem}

\begin{proof}
Rotation and translation do not impact the quantity of the left-hand side of \eqref{eq:affine_inv}. Thus, we may assume without loss of generality that the center of $\Omega$ is the origin and that $\phi(0,0) =0$, $\nabla \phi(0,0) = (0,0)$. 

By keeping track of the scaling, we have $\hat G (\xi,\eta) = \hat F (T_\Omega(\xi),s\eta)$. Direct computation shows that
$$
G(x) = \frac{1}{s|\det T_\Omega|} F(T_\Omega^{-t}(x_1,x_2),s^{-1}x_3),
$$
where $T_\Omega^{-t}$ is the inverse transpose of $T_\Omega$.
Therefore, we have
\begin{equation}\label{eq:4.3}
    \|F\|_{L^4(\R^3)} = \left(s|\det T_\Omega|\right) ^{1-1/4}\|G\|_{L^4(\R^3)},
\end{equation}

On the other hand, $|(\det D^2 \phi) (T_{\Omega} \cdot) | |\det T_\Omega|^2 = s^2|\det D^2 \bar{\phi} (\cdot)|$. Hence, 
\begin{equation}\label{eq:4.4}
    \|\hat F\|_{L^2(dM^{\phi})} = (s^{-2} |\det T_\Omega|^2)^{(1/4)(1/2)} (s|\det T_{\Omega}|)^{1/2} \|\hat G\|_{L^2(dM^{\bar\phi})}.
\end{equation}
Combining \eqref{eq:4.3} and \eqref{eq:4.4}, we obtain \eqref{eq:affine_inv} as desired.
\end{proof}

Now, after rescaling, $(\phi,\sigma',R,\varepsilon')$-admissible sets $\Omega$ are in either of the following situations:
\begin{enumerate}
    \item $\sup_{[-1,1]^2}|\det D^2 \phi_{T_\Omega}| \lesssim \sigma |\Omega|^2 \lesssim_{\varepsilon'} R^{-2}$ and $\supp  \hat{G} \subset \mathcal{N}_{R^{-1}}^{\phi_{T_\Omega}}([-1,1]^2)$. In this case $s=1$.
    \item $\|\bar{\phi}_{T_\Omega}\| \sim_{\varepsilon'} 1$ and $\supp  \hat{G} \subset \mathcal{N}_{(sR)^{-1}}^{\bar\phi_{T_\Omega}}([-1,1]^2)$, where $s = \|\phi_{T_\Omega}\|$. $\det D^2 \bar{\phi}_{T_\Omega}\sim_{\varepsilon'} 1$ over $[-1,1]^2$.
\end{enumerate}

The first case can be estimated by Hausdorff-Young and H\"older's inequalities:
$$
    \|G\|_{L^4(\R^3)} 
    \lesssim \|\hat G\|_{L^{4/3}(\R^3)} 
    \lesssim (R^{-1})^{3/4-1/2} (\sigma |\Omega|^2)^{1/8} \|\hat G\|_{L^2(dM^{\phi_{T_\Omega}})} \lesssim_{\varepsilon'} R^{-1/2} \|\hat G\|_{L^2(dM^{\phi_{T_\Omega}})}.
$$
The second-to-last inequality follows from the facts that the support of $\hat{G}$ is of size $\lesssim R^{-1}$, and 
$$
d\xi d\eta \leq  \sup_{[-1,1]^2}|\det D^2 \phi_{T_\Omega}|^{1/4} dM^{\phi_{T_\Omega}}.
$$
Therefore, the right hand side of \eqref{eq:affine_inv} is bounded, and thus we obtain \eqref{eq:4.1} as desired.

We consider the second case. By Theorem \ref{eq:T-S restriction} in an equivalent formulation, we have
$$
\|G\|_{L^4(\R^3)} \lesssim_{\varepsilon'} (sR)^{-1/2}\|\hat G\|_{L^2(d\xi d\eta)} \lesssim_{\varepsilon'} (sR)^{-1/2}\|\hat G\|_{L^2(dM^{\bar \phi_{T_\Omega}})}.
$$
See Proposition 1.27 and Exercise 1.34 of \cite{Demeter2020}. Again, the right-hand side of \eqref{eq:affine_inv} is bounded. We obtain \eqref{eq:4.1}. This finishes the proof of Proposition \ref{prop:est_on_each_piece}.

\section{Proof of Theorem \ref{thm:dec_main}}\label{sec:proof_dec}

In this section, we prove Theorem \ref{thm:dec_main}. All implicit constants in this section are assumed to depend on the degree $d$.

\subsection{An initial step}
The first step is to decouple $\mathcal{N}_{R^{-1}}^\phi([-1,1]^2)$ into sets having either tiny $\det D^2 \phi$, or constant $\det D^2 \phi$. To do so, we first project our surface onto $[-1,1]^2$ and decouple the square into parallelograms by the following proposition. The strategy is similar to the generalised 2D decoupling in \cite{LY2021.2} (see Proposition \ref{prop:2D_general_uniform_no_size} below). It is worth noting that we will decouple the square $[-1,1]^2$ into parallelograms of different size or orientation. Otherwise, it is impossible to decouple the square non-trivially. See, for instance, Proposition 9.5 of \cite{Demeter2020}. 

To explain with a concrete example, we consider $P=\xi_2-\xi_1^2$. We first decouple $[-1,1]^2$ trivially, at a cost of $O(\log R)$, into level sets $\{|P| \sim \sigma\}$, for dyadic $\sigma \in (R^{-2},1]$, and $\{|P| \leq R^{-2}\}$. In this special case, $\{|P| \sim \sigma\}$ is a $O(\sigma)$ neighborhood of a parabola, and thus can be $\ell^2$ decoupled into smaller parallelograms of dimensions roughly $\sigma^{1/2} \times \sigma$. The above process illustrates how one may decouple $[-1,1]^2$ into parallelograms adapted to level sets of $P$.

\begin{prop}[Generalised 2D decoupling inequality with size estimates]\label{prop:low_dim_dec}
Let $2\leq p \leq 6$, $R \gg 1$, $\lambda \in (0,1)$ be a dyadic numbers, and $P:\R^2 \to \R$ be a bounded polynomial of degree at most $d$. For each dyadic number $\sigma \in [\lambda,1]$, there exist a family $\mathcal{P}_{\sigma,\lambda}^0 = \mathcal{P}_{\sigma,\lambda}^0(R,\phi)$ of parallelograms such that the following statements hold:
\begin{enumerate}
    \item $\mathcal{P}_\lambda^0:= \cup_{\sigma} \mathcal P_{\sigma,\lambda}^0$, where the union is over dyadic numbers $\sigma \in [\lambda,1]$, covers $[-1,1]^2$ in the sense that
    $$
        \sum_{\Omega_0 \in \mathcal P^0_\lambda} 1_{\Omega_0}(\xi) \geq 1 \quad \text{for }\xi \in [-1,1]^2;
    $$
    \item $100\mathcal{P}_\lambda^0$ has bounded overlap in the sense that $\sum_{\Omega_0 \in \mathcal{P}_\lambda^0} 1_{100\Omega_0} \lesssim_d 1$;
    \item\label{item:size_estimate} for each $\Omega \in \mathcal{P}^0_{\sigma,\lambda}$, at least one of the following holds:
    \begin{enumerate}
        \item\label{item:a} $\lambda \leq \sigma \leq 1$, $|P| \sim \sigma$ over $10\Omega$, and the width of $\Omega$ is at least $\max\{\sigma,R^{-1}\}$;
        \item\label{item:b} $\lambda \leq \sigma \leq R^{-1}$, $|P| \lesssim \sigma$ over $10\Omega$, and the width of $\Omega$ is $\sim R^{-1}$.
        \item\label{item:c} $\sigma=\lambda$, $\sup_{10\Omega}|P| \lesssim \sigma$, and the width of $\Omega$ is at least $\max \{ \sigma, R^{-1}\}$.
    \end{enumerate}
    \item For $\lambda' < \lambda < \sigma,$ $\P_{\sigma,\lambda}^0 = \P_{\sigma,\lambda'}^0$.
    \item the following 2D $\ell^2(L^p)$ decoupling inequality holds: for any $f:\R^2 \to \C$ whose Fourier transform is supported on $\cup_{\Omega \in \mathcal{P}^0_{\sigma,\lambda}} \Omega$, we have, for each $\varepsilon>0$,
    \begin{equation}\label{eq:low_dim_dec}
        \|f\|_{L^p(\R^2)} \lesssim_{\varepsilon,d} \sigma^{-\varepsilon} \left ( \sum_{\Omega \in \mathcal{P}^0_{\sigma,\lambda}} \|f_{\Omega}\|_{L^p(\R^2)}^2 \right)^{1/2},
    \end{equation}
    where $f_\Omega$ is the frequency projection of $f$ onto $\Omega$, defined by $\hat f_\Omega = \hat{f}1_\Omega$.
\end{enumerate}
\end{prop}

Before proceeding, we explain the parameter $\lambda$ and the size estimate \ref{item:size_estimate}. In application, we set $P = \det D^2\phi$ and $\lambda = R^{-2}$. By cylindrical decoupling, Proposition \ref{prop:low_dim_dec} allows us to decouple our surface into rectangular sets on which $|\det D^2 \phi| \lesssim \sigma$, for some $\sigma \gtrsim \lambda = R^{-2}$. Firstly, size estimate item (\ref{item:a}) guarantees that the rescaled function $\phi_{T_\Omega}$, defined in \eqref{eq:phi_T_Omega}, has Hessian determinant $\gtrsim \sigma^{O(1)}$. This estimate in Hessian determinant is crucial because we will induct on the Hessian determinant in later steps. Secondly, size estimate item (\ref{item:b}) acts as a stopping condition. While $|\det D^2 \phi|$ may be bigger than $R^{-2}$ somewhere on the set, the width condition ensures $\phi$ is ``one-dimensional'' on the set. This is because the change of $\phi$ is bounded by $O(R^{-1})$ along the shorter side of $\Omega$. It suffices to decouple along the longer side by lower dimensional decoupling. Finally, sets satisfying size estimate item (\ref{item:c}) can be put directly into the collection $\mathcal P$ in Theorem \ref{thm:dec_main}.

In what follows, the implicit constant in $\{|P| \sim \sigma\}$ may change. This is allowed because we can decouple the larger set into $O(1)$ many sets with a loss of $O(1)$, by triangle inequality and H\"older's inequality. 

The proof of Proposition \ref{prop:low_dim_dec} relies on the following result on the generalized 2D uniform decoupling for polynomials:

\begin{prop}[Theorem 3.1, \cite{LY2021.2}]\label{prop:2D_general_uniform_no_size}
Let $\varepsilon>0$, $2\leq p\leq 6$, $0<\delta<1$, $2 \leq d\in \mathbb N$ and parallelogram $\Omega_0 \subset[-1,1]^2$. For any bounded polynomial $P: \R^2 \to \R$ of degree at most $d$, there exists a cover $\mathcal T_\delta$ of the set
$$
\{(\xi_1,\xi_2)\in \Omega_0:|P(\xi_1,\xi_2)|<\delta\}
$$
by rectangles $\Omega$ such that the following holds:
\begin{enumerate}
    \item\label{item:1;prop:2D_general_uniform_no_size} $|P|\lesssim \delta$ on $10\Omega$;
    \item $\mathcal T_\delta$ has bounded overlap in the sense that $\sum_{\Omega\in \mathcal T_\delta}1_{100\Omega}\lesssim 1$;
    \item\label{item:3;prop:2D_general_uniform_no_size} if in addition $|\nabla P| \sim \kappa$ over $10\Omega$ and the width of $\Omega_0$ is at least $\kappa^{-1}\delta$, then the width of $\Omega$ is $\sim \kappa^{-1}\delta$;
    \item\label{item:4;prop:2D_general_uniform_no_size} the following $\ell^2(L^p)$ decoupling inequality holds: for any function $f:\R^2\to \C$ Fourier supported on $ \{(\xi_1,\xi_2)\in \Omega_0:|P(\xi_1,\xi_2)|<\delta\}$,
    \begin{equation}\label{eqn_2D_general_decoupling}
         \norm f_{L^p(\R^2)}\leq C_{\varepsilon,d}\delta^{-\varepsilon}\left(\sum_{\Omega\in \mathcal T_\delta}\norm{f_\Omega}^2_{L^p(\R^2)}\right)^{\frac 1 2}.
    \end{equation}
    where $f_\Omega$ is the frequency projection of $f$ onto $\Omega$, defined by $\hat f_\Omega = \hat{f}1_\Omega$.
\end{enumerate}
\end{prop}

For item (\ref{item:3;prop:2D_general_uniform_no_size}) in Proposition \ref{prop:2D_general_uniform_no_size}, it is not explicitly stated in the Theorem 3.1 of \cite{LY2021.2}. We briefly explain why we have the size estimate. Suppose that $|\nabla P| \sim \kappa$ over $10\Omega$. By mean value theorem, $\{|P|<\delta\}\cap \Omega_0$ is within the $O(\kappa^{-1}\delta)$ neighborhood of $\{P=0\} \cap \Omega_0$. Since $P$ is a polynomial, $\{P=0\}$ can be represented by a $O(1)$ many disjoint smooth function $g_k$, up to a rotation. See Figure 1 and Proposition 4.5 of \cite{LY2021.2}. $O(\kappa^{-1}\delta)$ neighborhood of a $g_k$ can then be decoupled using two dimensional techniques and Pramanik-Seeger iteration. For details, see Proposition 4.6 of \cite{LY2021.2}.

We will prove Proposition \ref{prop:low_dim_dec} by induction on the degree of the polynomial $P$. For now, let us pretend that $|\nabla P|$ is a polynomial of one degree lower for the induction purpose. One key observation is that if $|\nabla P| \lesssim \sigma$ over some parallelogram $\Omega_0$, either $|P| \lesssim \sigma$ or $|P| \sim \sigma'$, for some $\sigma'>\sigma$, on $\Omega_0$. Inherited from the size estimate in the induction hypothesis, the parallelogram $\Omega_0$ automatically qualifies to stay inside $\P_{\lambda}$. Thus, to decouple the sub-level set $\{|P| \lesssim \sigma\}$, we first decouple the sets $\{|\nabla P| \lesssim \sigma\}$ and $\{|\nabla P| \sim \sigma'\}$ for $\sigma'>\sigma$. 

\begin{proof}[Proof of Proposition \ref{prop:low_dim_dec}] 

Proposition \ref{prop:low_dim_dec} is clear when $P$ is a constant function. By inducting the degree of the polynomial, we may assume that there exist families $\mathcal{P}_{\sigma_1,\sigma/C}^{\xi_1}$,$\mathcal{P}_{\sigma_2,\sigma/C}^{\xi_2}$, satisfying the five statements for functions $\partial_{\xi_1}P,\partial_{\xi_2}P$ replacing $P$ in Proposition \ref{prop:low_dim_dec} respectively.

Let $\lambda$ be a dyadic number. We now construct the families $\mathcal{P}_{\sigma,\lambda}^0$, for $\lambda\leq \sigma\leq 1$, to fulfill statements 1 and 3 in Proposition \ref{prop:low_dim_dec}. Other statements will be checked later. In particular, $\cup _\sigma \mathcal{P}_{\sigma,\lambda}^0$ needs to cover $[-1,1]^2$. We now describe this covering.

First, we cover $ [-1,1]^2$ by sets $\{ |P| \sim \sigma\} $ for dyadic numbers $\sigma > \lambda$ and $\{|P| \leq \lambda \}$.

If $|P(\xi)| \sim \sigma$ for some dyadic number $\sigma > \lambda$, $\xi \in \Omega_{1} \cap \Omega_2$ for some $\Omega_i \in \mathcal{P}^{\xi_i}_{\sigma_i,\sigma/C}$, $\sigma_i \geq \sigma/C$, $i=1,2$. Otherwise, $|P(\xi)| < \sigma= \lambda$ and $\xi \in \Omega_{1} \cap \Omega_2$ for some $\Omega_i \in \mathcal{P}^{\xi_i}_{\sigma_i,\sigma/C}$, $\sigma_i \geq \sigma/C  = \lambda/C$, $i=1,2$. These cases are indexed by dyadic $\sigma \in [\lambda,1]$.

We may restrict our attentions to $  \{ |P| \sim \sigma\}\cap \Omega_1 \cap \Omega_2  $ and $  \{ |P| \leq \sigma = \lambda\} \cap \Omega_1 \cap \Omega_2  $. The width estimate follows because
$$
\text{width}(\Omega_1 \cap \Omega_2) \gtrsim \max\{\min\{\sigma_1,\sigma_2\},R^{-1}\} \gtrsim \max\{ \sigma,R^{-1}\}.
$$
Without loss of generality, we may assume that $\sigma_1\leq \sigma_2$. We now consider the following 3 cases corresponding to the items (\ref{item:a}), (\ref{item:b}), and (\ref{item:c}), one of which $\Omega_2$ satisfies by our induction hypothesis. 

\underline{Case 1: $|\partial_{\xi_2}P| \sim \sigma_2 $ over $10\Omega_2$ for some $\sigma_2 \in [\sigma/C , 1]$.}

If $\sigma \in (\max\{\lambda,\sigma_2 R^{-1}\}, C\sigma_2] $ , the set $\{ |P \pm \sigma|<\sigma/C\}\cap \Omega_1\cap \Omega_2$ can be decoupled into 
$O_d(1)$-bounded overlapping parallelograms $\Omega$ of width $\gtrsim \sigma_2^{-1}\sigma \geq \max\{\sigma,R^{-1}\}$ by Proposition \ref{prop:2D_general_uniform_no_size}. We put these $\Omega$, which covers $  \{ |P| \sim \sigma\}\cap \Omega_1 \cap \Omega_2 $, in $\mathcal{P}_{\sigma,\lambda}^0$. They satisfy item (\ref{item:a}).

If $\sigma = \max\{\lambda, \sigma_2R^{-1} \}$  by Proposition \ref{prop:2D_general_uniform_no_size} again, the set $\{ |P|<C\sigma\}\cap \Omega_1\cap \Omega_2$ can be decoupled into $O_d(1)$-bounded overlapping parallelograms $\Omega$. We put these $\Omega$, which covers $  \{ |P| \sim \sigma\}\cap \Omega_1 \cap \Omega_2 $, in $\mathcal{P}_{\sigma,\lambda}^0$. If $\sigma=\sigma_2R^{-1}$, $\Omega$ has width $\sim R^{-1}$ and satisfies item (\ref{item:b}). If $\sigma = \lambda$, $\Omega$ satisfies item (\ref{item:c}) and covers $\{|P| < \lambda\} \cap \Omega_1 \cap \Omega_2$.

Since $\sigma_2 \geq \sigma/C$ and $\sigma \geq \lambda$, we have constructed a covering in this case.

\underline{Case 2: $|\partial_{\xi_2}P| \lesssim \sigma_2$ over $10\Omega_2$ for some $\sigma_2 \in [\sigma/C, R^{-1}] $, and $\Omega_2$ has width $\sim R^{-1}$.}

If $\sigma \in (\lambda,C \sigma_2]$  and $\sup_{10(\Omega_1\cap \Omega_2)}|P| \sim \sigma$, we put the set $\Omega_1\cap \Omega_2 = \{|P| \sim \sigma\} \cap \Omega_1\cap \Omega_2 $ in $\mathcal{P}_{\sigma,\lambda}^0$. Recall that $\Omega_1$ has width 
at least $\max\{\sigma,R^{-1}\}$. Therefore, $\Omega_1\cap \Omega_2$ has width $\sim R^{-1}$ and hence satisfies item (\ref{item:b}).

If $\sigma = \lambda$ and $\sup_{10(\Omega_1\cap \Omega_2)}|P| \lesssim \sigma =\lambda$, we also put the set $\Omega_1\cap \Omega_2 = \{|P| < \lambda\} \cap \Omega_1\cap \Omega_2$ in $\mathcal{P}_{\sigma,\lambda}^0$. For the exact same reason, $\Omega_1\cap \Omega_2$ satisfies item (\ref{item:b}).

For all other cases, $\Omega_1\cap \Omega_2$ has negligible intersection with $\{|P| \sim \sigma\}$ or $\{|P| < \sigma = \lambda\}$ because $|\nabla P | \leq \sigma/C$.

Note that the set $\Omega_1\cap \Omega_2$ is in $O(1)$ many families $\P_{\sigma,\lambda}^0$ for fixed $\lambda$. 

\underline{Case 3: $|\partial_{\xi_2}P| \lesssim \sigma/C$ over $10\Omega_2$.}

In this case, either $|P| \sim \sigma$  or $|P| < \lambda$ over $10(\Omega_1\cap \Omega_2)$.

If $|P| \sim \sigma$ on $10 (\Omega_1\cap \Omega_2)$, $\Omega_1\cap \Omega_2$ is put in $\P^{0}_{\sigma,\lambda}$ and it satisfies item (\ref{item:a}). 

If $\sigma= \lambda$ and $\sup_{10 (\Omega_1\cap \Omega_2)} |P| <\lambda$, $\Omega_1\cap \Omega_2$ is put in $\mathcal{P}_{\lambda,\lambda}^0$ and satisfies item (\ref{item:c}). 

\underline{Verification of statements 2, 4 and 5 in Proposition \ref{prop:low_dim_dec}}

\textit{\framebox{Statement 2:} $100\mathcal{P}_\lambda^0$ has bounded overlap in the sense that $\sum_{\Omega_0 \in \mathcal{P}_\lambda^0} 1_{100\Omega_0} \lesssim_d 1$.}

By induction hypothesis, $100\P_\lambda^{\xi_1}$ and $100\P_\lambda^{\xi_2}$ have $C_{d-1}$-bounded overlap. Hence
$$
\sum_{\Omega_{i} \in \P_\lambda^{\xi_i} , i=1,2} 1_{100\Omega_1 \cap 100\Omega_2} \lesssim C_{d-1}^2.
$$

After decoupled into $O_d(1)$-bounded overlapping parallelograms $\Omega_0$ if necessary, each set $\Omega_1 \cap \Omega_2$ is put into $O(1)$ many families $\P_{\sigma,\lambda}^0$. Therefore,
$$
\sum_{\substack{\Omega_0\in \P_\lambda^0\\ \Omega_0 \subset \Omega_1 \cap \Omega_2} } 1_{100\Omega_0} \leq c_d,
$$
and
$$
\sum_{\Omega_0\in \P_\lambda^0 } 1_{100\Omega_0} \leq C_d := c_d C_{d-1}^2.
$$
\\
\textit{\framebox{Statement 4:}For $\lambda ' < \lambda< \sigma$, $\mathcal{P}_{\sigma,\lambda}^0 = \mathcal{P}_{\sigma,\lambda'}^0$.}

Let $\lambda'<\lambda<\sigma$. The key observation is that  sets $\Omega \in \P_{\lambda'}^0 \setminus \P_{\lambda}^0$ comes from the decomposition of sets in $\P_{\lambda,\lambda}^0$.

Suppose that  $\Omega \in \P_{\sigma,\lambda}^0$. Note that we apply the induction hypothesis to get the same families $\P^{\xi_i}_{\sigma_i,\sigma/C}$, $i=1,2$. If $\Omega$ comes from Case 1, it is decoupled from either the set $\{ |P \pm \sigma| < \sigma/C\} \cap \Omega_1 \cap \Omega_2$, or the set $\{ |P| < \sigma\} \cap \Omega_1 \cap \Omega_2$. Thus, we get the exact same families of $\Omega$ from these $\Omega_1\cap \Omega_2$. Otherwise, $\Omega$ comes from Case 2 or 3, where $\Omega= \Omega_1\cap \Omega_2$. Thus, we see that in any case, $\Omega \in \P_{\sigma,\lambda'}^0$. The other direction is similar.\\
\textit{\framebox{Statement 5:}the following 2D $\ell^2(L^p)$ decoupling inequality holds: for any $f:\R^2 \to \C$ whose Fourier transform is supported on $\bigcup_{\Omega \in \mathcal{P}^0_{\sigma,\lambda}}\Omega$, we have
    \begin{equation}
        \|f\|_{L^p(\R^2)} \lesssim_{\varepsilon,d} \sigma^{-\varepsilon} \left ( \sum_{\Omega \in \mathcal{P}^0_{\sigma,\lambda}} \|f_{\Omega}\|_{L^p(\R^2)}^2 \right)^{1/2},
    \end{equation}
    where $f_\Omega$ is the frequency projection of $f$ onto $\Omega$, defined by $\hat{f_\Omega}= \hat f 1_{\Omega}$.}

It suffices to consider the decoupling constant. Note that the number of iterations is $O(d)$. Also, all $\Omega$ put in $\P_{\sigma,\lambda}^0$ are from $\P^{\xi_i}_{\sigma_i,\sigma/C}$, $i=1,2$, and $\sigma_i \geq \sigma$. The decoupling loss from the inductive step is therefore $\lesssim_\varepsilon (\sigma_1\sigma_2)^{-O_d(\varepsilon)} \lesssim \sigma^{-O_d(\varepsilon)}$. The decoupling loss from Proposition \ref{prop:2D_general_uniform_no_size} is $\lesssim_\varepsilon \sigma^{-\varepsilon}$. The total loss is $O_\varepsilon(\sigma^{-O_d(\varepsilon)})$. Since $\varepsilon$ is arbitrary, we have shown that the decoupling constant is $O_\varepsilon(\sigma^{-\varepsilon})$ as desired.

This finishes the proof of Proposition \ref{prop:low_dim_dec}.
\end{proof}

From now on, let $P= \det D^2\phi$. With Proposition \ref{prop:low_dim_dec} in hand, we apply the cylindrical decoupling (See Exercise 9.65 in \cite{Demeter2020}, or Proposition 6.2 of \cite{LY22} for details) to get the decoupling inequality
\begin{equation}\label{eq:low_dim_dec_global}
    \|F\|_{L^p(\R^3)} \lesssim_{\varepsilon,d} \sigma^{-\varepsilon} \left ( \sum_{\Omega_0 \in \mathcal{P}^0_\sigma} \|F_{\Omega_0}\|_{L^p(\R^3)}^2 \right)^{1/2},
\end{equation}
whenever $F:\R^3 \to \C$ is Fourier supported on $\mathcal{N}_{R^{-1}}(\cup_{\Omega_0 \in \mathcal{P}_\sigma^0}\Omega)$.

Note that if $\Omega_0$ satisfies item (\ref{item:c}), then it qualifies to be in the families $\mathcal{P}_{R^{-2}}$ in Theorem \ref{thm:dec_main}. Thus, it suffices to further decouple $\Omega_0$ that satisfies item (\ref{item:a}) or (\ref{item:b}). In the rest of this section, we will construct an $O_\varepsilon(\sigma^{-\varepsilon})$-bounded overlapping $(\phi,\sigma,R,\varepsilon)$-admissible covering $\P_\sigma^{\Omega_0}$ of each $\Omega_0$ by parallelograms of widths at least $R^{-1}$ satisfying the decoupling inequality:
\begin{equation}\label{eq:dec_Omega_0}
    \|F_{\Omega_0}\|_{L^p(\R^3)} \lesssim_{\varepsilon,d} \sigma^{-O(\varepsilon)} \left ( \sum_{\Omega \in \mathcal{P}^{\Omega_0}_\sigma} \|F_{\Omega}\|_{L^p(\R^3)}^2 \right)^{1/2}.
\end{equation}

Assume the families $\P_\sigma^{\Omega_0}$ for now. Let $\P_\sigma$ be the union of $\P_\sigma^{\Omega_0}$ over all $\Omega_0\in \P^0_\sigma$. Theorem \ref{thm:dec_main} follows immediately by \eqref{eq:low_dim_dec_global} and the properties of $\Omega_0 \in \P^{0}_\sigma$ in Proposition \ref{prop:low_dim_dec}.

\subsection{Case (\ref{item:b})}\label{subsec:item b} Let $\Omega_0\in \mathcal{P}^{0}_\sigma$ be a parallelogram satisfying (\ref{item:b}). We will decouple $\Omega_0$ into $(\phi,\sigma,R,\varepsilon)$-admissible parallelograms $\Omega$ satisfying $\|\phi_{T_\Omega}\|\lesssim R^{-1}.$ 

Let $T$ be the composition of rotation and translation that sends $\Omega_0$ to $\Tilde{\Omega}_0$, where $\Tilde{\Omega}_0$ is centered at $0$ and has the longer side parallel to $\xi_1$-axis. Since the width of $\Omega_0$ is $\sim R^{-1}$,
$$\phi\circ T^{-1}(\xi_1,\xi_2) = \phi\circ T^{-1}(\xi_1,0) + O(R^{-1}).$$ 
By cylindrical decoupling, it suffices to decouple the $R^{-1}$ neighborhood of $A(\xi_1):=\phi\circ T^{-1}(\xi_1,0)$. This can be done by the following uniform decoupling theorem for single variable polynomials.

\begin{thm}[Theorem 1.4 and Proposition 3.1, \cite{Yang2}]\label{thm:2D_uniform_dec_Yang}
For each $2\leq d \in \mathbb{N}$, $2\leq p\leq 6$, $\varepsilon>0$, there is a constant $ C_{\varepsilon,d}$ such that the following is true. For any $0< \delta\leq 1$, and any bounded polynomial $P:\R \to \R$ of degree at most $d$, there exists a partition $\mathcal{P}$ of $[-2,2]$ such that
\begin{enumerate}
    \item\label{item:1;thm:2D_uniform_dec_Yang} each $I \in \mathcal{P}$ satisfies $|I|\geq \delta^{1/2}$;
    \item\label{item:2;thm:2D_uniform_dec_Yang} each $I \in \mathcal{P}$ is $\delta$-flat, i.e.:
    \begin{equation}\label{eq:thm:2D_uniform_dec_Yang_flatness}
        \sup_{\xi,\xi'\in I} |P(\xi)-P(\xi') - P'(\xi')(\xi-\xi')|\leq \delta.
    \end{equation}
    \item $\sum_{I\in \mathcal{P}}1_{100I} \lesssim_d 1$.
    \item the following $\ell^2(L^p)$ decoupling inequality holds: for any $f\in \mathcal{S}$ Fourier supported on $\mathcal{N}_\delta^P([-2,2])$, we have
    \begin{equation}\label{eq:thm:2D_uniform_dec_Yang}
        \|f\|_{L^p(\R^2)} \leq C_{\varepsilon,d} \delta^{-\varepsilon} \left( \sum_{I\in \mathcal{P}} \|f_I\|_{L^p(\R^2)}^2 \right),
    \end{equation}
    where $f_I$ is the Fourier restriction of $f$ onto $I\times \R$.
\end{enumerate}

Moreover, for any partition $\mathcal{P}'$ of which $\mathcal{P}$ is a sub-partition, all properties except item (\ref{item:2;thm:2D_uniform_dec_Yang}) continue to hold for $\mathcal{P}'$.
\end{thm}

We remark that Proposition \ref{prop:2D_general_uniform_no_size} is more general than Theorem \ref{thm:2D_uniform_dec_Yang} except for the size estimates. So we state the exact theorem we need from \cite{Yang2} instead.

Applying Theorem \ref{thm:2D_uniform_dec_Yang} on the single variate polynomial $A$ and $\delta=R^{-1}$, we get a partition $\mathcal{P}^{\Omega_0}$ satisfying the four properties in Theorem \ref{thm:2D_uniform_dec_Yang}. Let $\mathcal{P}^{\Omega_0}_\sigma:=\{T(I \times \R) \cap \Omega_0 : I \in \mathcal{P}^{\Omega_0}\}$. By cylindrical decoupling, we get 
\begin{equation}
    \|F_{\Omega_0}\|_{L^p(\R^3)} \lesssim_{\varepsilon,d} R^{\varepsilon} \left ( \sum_{\Omega \in \mathcal{P}^{\Omega_0}_\sigma} \|F_{\Omega}\|_{L^p(\R^3)}^2 \right)^{1/2}.
\end{equation}

It is clear from \eqref{eq:thm:2D_uniform_dec_Yang_flatness} that $\|\phi_{T_\Omega}\|\leq R^{-1}$. Thus, we obtain a decoupling of $\Omega_0$ into parallelograms $\Omega$ that is $(\phi,\sigma,R,\varepsilon)$-admissible satisfying \eqref{eq:dec_Omega_0} because $R^{-2} \lesssim \sigma \lesssim R^{-1}$ and the width of $\Omega_0$ is $\sim R^{-1}$.

\subsection{Case (\ref{item:a})} We start with a parallelogram $\Omega_0$ satisfying (\ref{item:a}). The decoupling of $\Omega_0$ involves an induction on scale argument.

\subsubsection{The induction step:}
Let $T_\Omega$ be the invertible affine map such that $T_\Omega([-1,1]^2) = \Omega$. Recall the following function in \eqref{eq:phi_T_Omega} and \eqref{eq:bar_phi_T_Omega}:
\begin{equation}
    \phi_{T_\Omega}:=\phi \circ T_\Omega - \nabla (\phi \circ T_\Omega)(0,0) \cdot \xi - \phi \circ T_\Omega(0,0)
\end{equation}
and its normalisation
\begin{equation}
    \bar{\phi}_{T_\Omega} = \frac{\phi_{T_\Omega}}{\|\phi_{T_\Omega}\|}.
\end{equation}

By construction, $\bar{\phi}_{T_\Omega}$ is a bounded polynomial of degree at most $d$. Note that the graph of $\bar{\phi}_{T_\Omega}$ is a translated, rotated and enlarged copy of the graph of $\phi$ over $\Omega$. Thus, the hessian determinant of $\bar \phi_{T_\Omega}$ is also dyadically a constant over $[-2,2]^2$.

We induct on the following quantity:
\begin{equation}
    H(\Omega):= \inf_{\xi \in [-1,1]^2}|\det D^2\bar{\phi}_{T_\Omega}(\xi)| \sim  \|\phi_{T_\Omega}\|^{-2}|\det T_\Omega|^{2}\sigma =   \|\phi_{T_\Omega}\|^{-2}|\Omega|^{2} \sigma \lesssim 1.
\end{equation}

By the size estimate (\ref{item:a}), $H(\Omega) \sim \|\phi_{T_\Omega}\|^{-2}|\Omega|^{2} \sigma \ge \sigma^5$, since $\Omega$ has width at least $\sigma$ and $|\Omega| \gtrsim \sigma^2$.

Recall the definition of $(\phi,\sigma,R,\varepsilon)$-admissible sets in Definition \ref{def:admissible}, $\Omega$ is admissible if $H(\Omega)\sim_\varepsilon 1$ or $\|\phi_{T_{\Omega}}\| \lesssim R^{-1}$. So our goal is to decouple inductively until we achieve either of the conditions for all decoupled pieces.

We need the following result on polynomials with a small Hessian determinant over $[-1,1]^2$.

\begin{prop}[Theorem 3.2, \cite{LY2021.2}]\label{thm_small_hessian}
For each $2\leq d \in \mathbb{N}$, there is a constant $\alpha=\alpha(d)\in (0,1]$ such that the following holds. Let $P:\R^2\to \R$ be a bounded polynomial of degree at most $d$, without linear terms. If $\|\det D^2 P\|\leq \nu\in (0,1)$, then there exists a rotation $\rho:\R^2\to \R^2$, bounded polynomials $A,B$ such that
$$
P(\xi)=A\circ\rho(\xi)+\nu^{\alpha}B\circ \rho(\xi),
$$
and $A$ is one-dimensional, i.e. for any $\xi\in \R^2$,
$$
A(\xi_1,\xi_2)=A(\xi_1,0).
$$
\end{prop}

Now, we are ready to state and prove the key induction step:

\begin{prop}\label{prop:induction_step}
Let $\varepsilon>0$, $R\geq 1$, $2 \leq d \in \mathbb{N}$, $2\leq p\leq 6$, $R^{-2}<\sigma\leq 1$, $\alpha$ as in Proposition \ref{thm_small_hessian}, and $\phi$ be a bounded polynomial of degree at most $d$. Let $\Omega \subset [-1,1]^2$ be a parallelogram such that $|\det D^2\phi|$ is dyadically a constant on $10\Omega$. Let $T_\Omega$,  $\bar{\phi}_{T_\Omega}$ and $H(\Omega)$ be as above. Suppose that $H(\Omega)\geq \sigma^5$. Then,  there exists a covering $\mathcal{P}^\Omega = \mathcal{P}^\Omega_{iter} \sqcup \mathcal{P}^\Omega_{stop}$ of parallelograms $\omega$ such that the following holds: 
\begin{enumerate}
    \item\label{item:1,induction_step} $100\mathcal{P}^\Omega$ has $O(1)$-bounded overlaps in the sense that 
    \begin{equation}\label{eq:induction_ep_bdd_overlap}
        \sum_{\omega \in \mathcal P^\Omega} 1_{100\omega} \lesssim_d 1;
    \end{equation}
    \item\label{item:2,induction_step} for each $\omega \in \mathcal{P}^\Omega_{iter}$, $H(\omega) \gtrsim_d H(\Omega)^{1-\alpha/2}$;
    \item \label{item:3,induction_step} for each $\omega \in \mathcal{P}^\Omega$, the width of $\omega$ is bounded below by $CH(\Omega)^{3\alpha/4}$ times of the width of $\Omega$ for some absolute constant $C$;
    \item \label{item:4,induction_step} for each $\omega \in \mathcal{P}^\Omega_{stop}$, the width of $\omega$ is $\sim R^{-1}$;
    \item\label{item:5,induction_step} each $\omega\in \mathcal{P}^\Omega$ is contained inside $\left(1+CH(\Omega)^{\alpha/d}\right)\Omega$ for some absolute constant $C$;
    \item\label{item:6,induction_step} the following $\ell^2(L^p)$ decoupling inequality holds: for any $F$ Fourier supported on $\mathcal{N}_{R^{-1}}^\phi(\Omega)$,
        \begin{equation}\label{eq:induction_dec_main}
        \|F\|_{L^p(\R^3)} \lesssim_{\varepsilon,d} H(\Omega)^{-\varepsilon} \left ( \sum_{\omega \in \mathcal P^\Omega} \|F_{\omega}\|_{L^p(\R^3)}^2 \right)^{1/2}.
        \end{equation}
\end{enumerate}
\end{prop}

\begin{proof}[Proof of Proposition \ref{prop:induction_step}]Since the smaller $\alpha$ is, the easier Proposition \ref{thm_small_hessian} holds. So we may, without loss of generality, let $\alpha<1/1000$. 

By affine invariance of decoupling inequality, it suffices to decouple $\mathcal N_{R^{-1}}^{\bar{\phi}_{T_\Omega}}([-1,1]^2)$.

Note that $\bar{\phi}_{T_\Omega}$ has hessian determinant $\sim H(\Omega)$ over $[-1,1]^2$. We apply Proposition \ref{thm_small_hessian} on $\bar{\phi}_{T_\Omega}$ to obtain a rotation $\rho:\R^2 \to \R^2$ and bounded polynomials $A$ and $B$ such that 
$$
\bar{\phi}_{T_\Omega}(\xi)-A\circ \rho(\xi)= H(\Omega)^{\alpha} (B\circ \rho)(\xi) = O(H(\Omega)^{\alpha}),
$$
where $A(\xi_1,\xi_2) = A(\xi_1,0)$.

Take $\delta = CH(\Omega)^\alpha+ R^{-1} \sim H(\Omega)^{\alpha}$ so that $\mathcal N_{R^{-1}}^{\bar{\phi}_{T_\Omega}}([-1,1]^2)$ lies in $\mathcal N_{\delta}^{A\circ \rho}([-1,1]^2)$. Hence, by cylindrical decoupling, along the direction of $\rho^{-1}(e_2)$, the decoupling of the set $\mathcal N_{\delta}^{A\circ \rho}([-1,1]^2)$ is reduced to the decoupling of the two-dimensional set
$$
\{(\xi_1,\xi_2)\in \rho([-1,1]^2) \subset [-2,2]^2: |\xi_2-A(\xi_1,0)|\leq \delta\}.
$$
To decouple this set, we apply Theorem \ref{thm:2D_uniform_dec_Yang} to obtain a partition of $[-2,2]$ into intervals $I\in \mathcal{P}$. For each $I$, we define $\omega_I$ to be a parallelogram such that $\rho \circ T_\Omega^{-1}(\omega_I)$ is the smallest rectangle of the form $I \times [a,b]$ that contains $(\rho \circ T_\Omega^{-1} (\Omega)) \cap (I\times \R)$. See Figure \ref{fig:cyl_dec} below. 

\begin{figure}[ht]
    \begin{tikzpicture}[scale=1.4]
    
    \fill[black!20] (-1.3,1.1) -- (-2.1,1.1) -- (-2.1,2.3/3-2) -- (-1.3,2.3/3-2)--(-1.3,1.1);\
    \draw[<->] (-1.1,2.3/3-2)--(-1.1,2.3/6-1+1.1/2) [anchor=west] node {$2h$}-- (-1.1-0.05,1.1);
    \draw[->,>=stealth] (-3,-1) node[left] {\large $\rho \circ T_\Omega^{-1}(\omega)$} to (-1.675,0.5);
    
    \draw[very thick] (-2,-1) -- (1,-2) -- (2,1) -- (-1,2)--(-2,-1);
    \draw[dotted] (-3,-1.5) -- (1.5,-3) -- (3,1.5) -- (-1.5,3) -- (-3,-1.5);
    
    \draw (-2.1,-3) -- (-2.1,3);
    \draw (-1.3,-3) -- (-1.3,3);
    \draw (-0.6,-3) -- (-0.6,3);
    \draw (0.4,-3) -- (0.4,3);
    \draw (1.2,-3) -- (1.2,3);
    \draw (2,-3) -- (2,3);
    \draw[<->,>=stealth] (1.2,-2.7) -- (1.6,-2.7) [anchor=north] node {$I$} -- (2,-2.7);
    \draw[<->,>=stealth] (0.4,-2.7) -- (0.8,-2.7) [anchor=north] node {$I$} -- (1.2,-2.7);
    \node at (-0.1,0) {\Large $...$};
    \draw[<->,>=stealth] (-1.3,-2.7) -- (-0.95,-2.7) [anchor=north] node {$I$} -- (-0.6,-2.7);
    \draw[<->,>=stealth] (-2.1,-2.7) -- (-1.7,-2.7) [anchor=north] node {$I$} -- (-1.3,-2.7);
    
    \node at (3.7,-.5) {\Large $|I|\leq H(\Omega)^{\alpha/d}$};
    \draw[very thick] (2.5,-2)--(2.5,-1.5) --(3,-1.5) -- (3,-2)--(2.5,-2);
    \node at (4.4,-1.75) {\Large $:  \rho \circ T_\Omega^{-1}(\Omega)$};

    \draw[<->,>=stealth] (-2,1.5) --  (-1.25,1.25) ;
    \draw[->,>=stealth] (-3,2) node[left] {\large $\sim H(\Omega)^{\alpha/d}$} to [bend left=60]  (-1.675,1.375);

    \end{tikzpicture}
    \caption{The cylindrical decoupling}
    \label{fig:cyl_dec}
\end{figure}
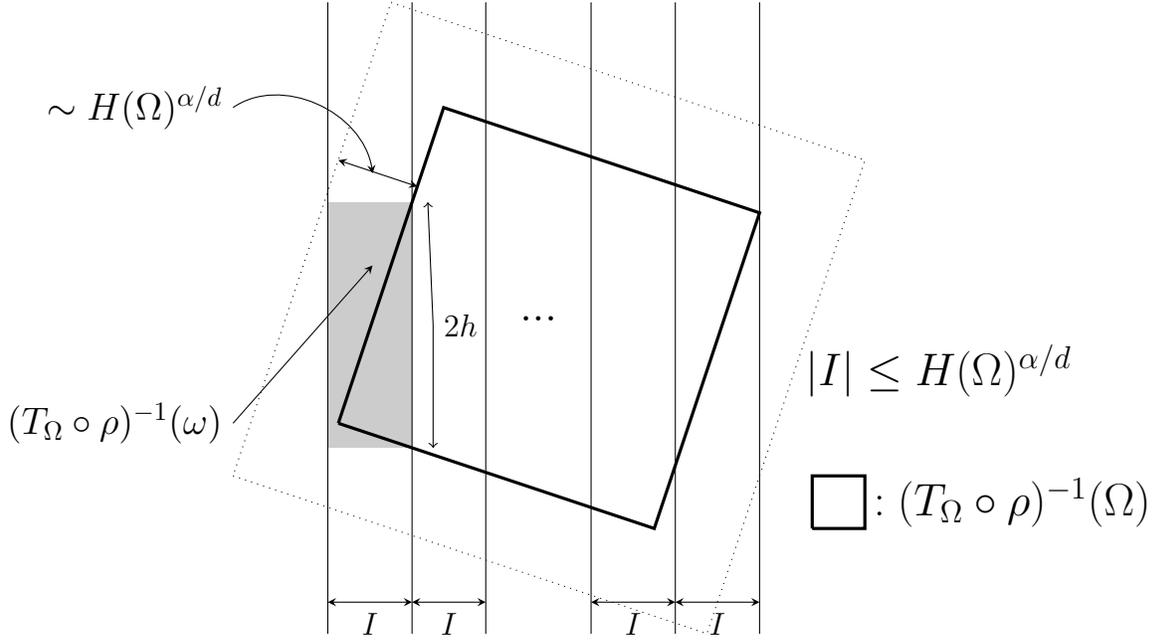

We remark that if $\rho $ is the identity map, no further treatment is needed as the decoupled collection of sets is a partition. In what follows, we deal with the technicalities caused by the rotation.

We merge intervals in $\mathcal{P}$ so that the width of $\omega_I$ is at least $R^{-1}$. If merges are involved, $\omega_I$ has width $R^{-1}$ and we put $\omega_I$ into $\mathcal{P}^\Omega_{stop}$. In this case, we also write $I=I'$. For each remaining $I$ satisfying \eqref{eq:thm:2D_uniform_dec_Yang_flatness}, we do the following.

Recall that $\rho \circ T_\Omega^{-1}(\omega_I)=I \times [a,b]$. By including a larger adjacent interval if necessary, we may assume without loss of generality that for $\xi_1$ in half of $I$, $\{\xi_1\}\times \R$ has non-empty intersection with $\rho \circ T_{\Omega}^{-1}(\Omega)$. Therefore, we have $2h:=b-a\ge |I|/2$. Let $T_0$ be the translation that maps $I\times [-h,h]$ to $I \times [a,b]$. By putting all terms involving $\xi_1$ only to $A \circ T_0$ if necessary, we may assume that
$$
\bar\phi_{T_\Omega} \circ \rho^{-1} \circ T_0 - A \circ T_0 = H(\Omega)^{\alpha}\xi_2\Tilde{B} 
$$
for some bounded polynomial $\Tilde{B}$. Since $|\xi_2|\leq h$ and $\delta \sim H(\Omega)^{\alpha}$, we have
$$
\bar\phi_{T_\Omega} \circ \rho^{-1} \circ T_0 - A \circ T_0 = O(\delta h).
$$

By size estimate (\ref{item:1;thm:2D_uniform_dec_Yang}) in Theorem \ref{thm:2D_uniform_dec_Yang}, $h \geq |I|/2 \gtrsim \delta^{1/2}$. Hence, $\delta h \gtrsim \delta^{3/2} \gg R^{-1}$. The vertical $R^{-1}$-neighborhood of $\bar\phi_{T_\Omega} \circ \rho^{-1} \circ T_0$ over $I \times [-h,h]$ is contained in the vertical $O(\delta h)$-neighborhood of $A \circ T_0$. We apply Proposition \ref{thm:2D_uniform_dec_Yang} again to decouple $I$ into $\delta h$-flat intervals $I'\in \mathcal{P}_I$. We define $\omega_{I'}$ to be a parallelogram such that $\rho \circ T_\Omega^{-1}(\omega_{I'}) = (I' \times \R) \cap \left (\rho \circ T_\Omega^{-1}(\omega_I)\right)$. Similarly, we merge the intervals $I'$ if necessary so that the width of $\omega_{I'}$ is at least $R^{-1}$ and put these $\omega_{I'}$ into $\mathcal{P}^\Omega_{stop}$. We put all remaining parallelograms $\omega=\omega_{I'}$ into $\mathcal{P}^\Omega_{iter}$. 

We now check Properties (\ref{item:1,induction_step}), (\ref{item:5,induction_step}) and (\ref{item:6,induction_step}) for this family $\mathcal{P}^\Omega = \mathcal{P}^\Omega_{iter} \sqcup \mathcal{P}^\Omega_{stop}$. Property (\ref{item:1,induction_step}) inherits from that of $\mathcal{P}$ and $\mathcal{P}_I$. Note that $\rho \circ T_\Omega^{-1}(\omega_I)$ is contained in the dotted rectangle, $(1+C H(\Omega)^{\alpha/d}) \rho \circ T_\Omega^{-1}(\Omega)$. Thus, rescaling back, we obtain property (\ref{item:5,induction_step}) for $\omega\subset \omega_I$. The decoupling inequality \eqref{eq:induction_dec_main} in property (\ref{item:6,induction_step}) follows from \eqref{eq:thm:2D_uniform_dec_Yang}, invariance of decoupling inequality under affine transformation, cylindrical decoupling and the fact that $\delta h \geq H(\Omega)^{3\alpha/2} \gg R^{-1}$.

Now, we check property (\ref{item:2,induction_step}). For $\delta h$-flat intervals $I'$, $\|\phi_{T_\omega}\| \leq \delta h \|\phi_{T_\Omega}\|$. On the other hand, by size estimate (\ref{item:1;thm:2D_uniform_dec_Yang}) in Theorem \ref{thm:2D_uniform_dec_Yang}, we have
$$|\omega| = |I'|(2h)|\Omega| \gtrsim h^{3/2} \delta^{1/2}|\Omega|,$$
and $h \gtrsim |I| \gtrsim \delta^{1/2}$.

Therefore,
\begin{align*}
    \|\phi_{T_{\omega}}\|^{-2}|\omega|^2 \sigma &\gtrsim (h\delta)^{-2}(h^{3/2}\delta^{1/2})^2\|\phi_{T_{\Omega}}\|^{-2}|{\Omega}|^2 \sigma \\ &\sim H(\Omega)h\delta^{-1} \gtrsim \delta^{-1/2}H(\Omega) \sim H(\Omega)^{1-\alpha/2},
\end{align*}
as desired.

Property (\ref{item:4,induction_step}) is immediate by the definition of the collection $\mathcal{P}^\Omega_{stop}$.

By size estimate (\ref{item:1;thm:2D_uniform_dec_Yang}) in Theorem \ref{thm:2D_uniform_dec_Yang}, $\delta h$-flat intervals $I'$ has length at least $(\delta h)^{1/2} \gtrsim \delta^{3/4} \sim H(\Omega)^{3\alpha/4}$. This proves property (\ref{item:3,induction_step}).

We have obtained all properties for the family $\mathcal{P}^\Omega$, hence Proposition \ref{prop:induction_step} is proved.
\end{proof}

\subsubsection{Induct on $H(\Omega)$} Let $K=K(\varepsilon,d)\gg 1$ to be determined. Let $\Omega_0 \in \mathcal{P}_\sigma^0$ be a parallelogram satisfying (\ref{item:a}). Note that $H(\Omega_0) \geq |\Omega_0|^2\sigma \geq \sigma^5$. We apply Proposition \ref{prop:induction_step} to decouple each $\Omega_0$ into $\Omega_1\in \mathcal{P}^{\Omega_0}$. We apply Proposition \ref{prop:induction_step} again to decouple each $\Omega_1\in\mathcal{P}^{\Omega_0}_{iter}$ into $\Omega_2\in \mathcal{P}^{\Omega_1}$, and so on. The process stops if $\Omega_{\bullet-1} \in \mathcal{P}^{\Omega_{\bullet-2}}_{stop}$, or $H(\Omega_{\bullet}) \geq 1/K$.

Suppose that there exists some $\Omega_{\bullet-1} \in \mathcal{P}^{\Omega_{\bullet-2}}_{stop}$ (i.e. the collection $\mathcal{P}^{\Omega_{\bullet-2}}_{stop}$ is non-empty). We will show momentarily that $R^{-1}\gtrsim \sigma^{21}$ in this case. Since the width of $\Omega_{\bullet-1}$ is $R^{-1}$, We repeat Section \ref{subsec:item b} again to get a decoupling of $\Omega_{\bullet-1}$ into $(\phi,\sigma,R,\varepsilon)$-admissible parallelograms $\Omega_{\bullet}$, with decoupling constant $R^{\varepsilon} \lesssim \sigma^{-21\varepsilon}$. 

The tree diagram below (Figure \ref{dia:tree_diagram}) describes the iteration process. The value next to each edge in the tree diagram represents an upper bound of the decoupling constant in that step. Let 
$\P_\sigma^{\Omega_0}$ be the collection of these sets $\Omega_\bullet$.

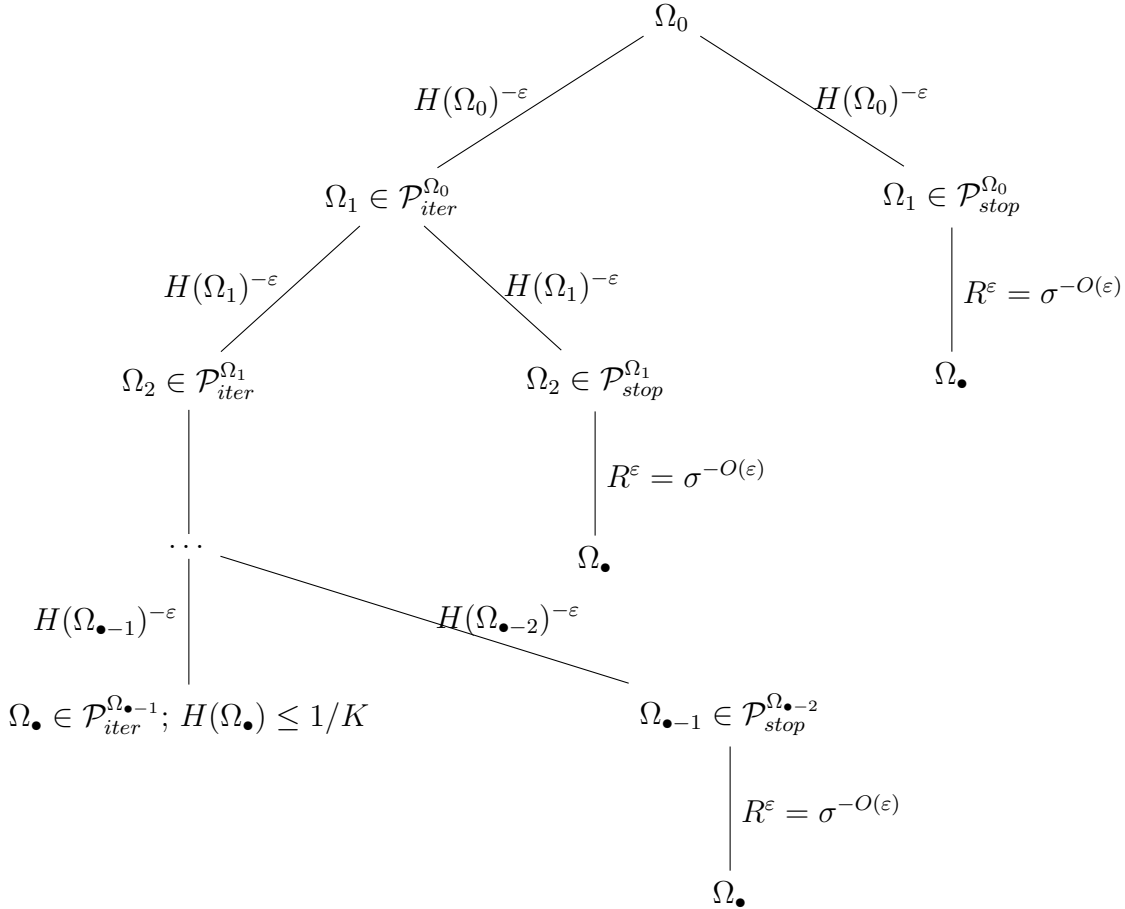
\begin{figure}[ht]
\centering
\begin{forest}
for tree={l sep=4em, s sep=8em, anchor=center}
    [$\Omega_0$,
        [$\Omega_1 \in \mathcal{P}^{\Omega_0}_{iter}$, edge label={node[midway,left]{$H(\Omega_0)^{-\varepsilon}$}}
            [$\Omega_2 \in \mathcal{P}^{\Omega_1}_{iter}$, edge label={node[midway,left]{$H(\Omega_1)^{-\varepsilon}$}}
                [\dots ,calign=first
                    [$\Omega_\bullet \in \mathcal{P}^{\Omega_{\bullet-1}}_{iter};$ $H(\Omega_\bullet)\geq 1/K$, edge label={node[midway,left]{$H(\Omega_{\bullet-1})^{-\varepsilon}$}}
                    ]
                    [$\Omega_{\bullet-1}\in \mathcal{P}^{\Omega_{\bullet-2}}_{stop}$, edge label={node[midway,right]{$H(\Omega_{\bullet-2})^{-\varepsilon}$}}
                        [$\Omega_\bullet$, edge label={node[midway,right]{$R^{\varepsilon} \lesssim \sigma^{-21\varepsilon}$}}
                        ]
                    ]
                ]
            ]
            [$\Omega_2 \in \mathcal{P}^{\Omega_1}_{stop}$, edge label={node[midway,right]{$H(\Omega_1)^{-\varepsilon}$}}
                [$\Omega_\bullet$ , edge label={node[midway,right]{$R^{\varepsilon} \lesssim \sigma^{-21\varepsilon}$}}
                ]
            ]
        ]
        [$\Omega_1 \in \mathcal{P}^{\Omega_0}_{stop}$, edge label={node[midway,right]{$H(\Omega_0)^{-\varepsilon}$}}
            [$\Omega_\bullet$ , edge label={node[midway,right]{$R^{\varepsilon} \lesssim \sigma^{-21\varepsilon}$}}
            ]
        ]
    ]
\end{forest}
\caption{Tree diagram for the induction}\label{dia:tree_diagram}
\end{figure}

Let $\{\Omega_{k} \}_{k=0}^N$ be a sequence of sets that forms a full branch of the tree diagram. For $k\leq N$, we may assume that $H(\Omega_{k-1})\leq 1/K$. Property (\ref{item:2,induction_step}) of Proposition \ref{prop:induction_step} implies that 
\begin{equation}\label{eq:HLowerbd}
    H(\Omega_{k}) \gtrsim_d H(\Omega_{k-1})^{1-\alpha/2} \implies H(\Omega_{k}) > c H(\Omega_{k-1})^{1-\alpha/4}
\end{equation}
for any absolute constant $c\geq 1$ by choosing $K$ large enough. We see that the maximum number of steps $N$ we need is $\leq c_\alpha \log(\sigma^{-1})/\log(K)$ to ensure that $H(\Omega_k)$ exceeds $1/K$.

We now prove that $\Omega_N$ is $(\phi,\sigma,R,\varepsilon)$-admissible. By property (\ref{item:3,induction_step}) of Proposition \ref{prop:induction_step}, 
$$
\text{width of }\Omega_{k} \geq  C H(\Omega_{k-1})^{3\alpha/4} \cdot (\text{width of }\Omega_{k-1}) \geq H(\Omega_{k-1})^{\alpha} \cdot (\text{width of }\Omega_{k-1})
$$
unless $\Omega_{k-1} \in \P_{stop}^{N-2}$. Therefore, the width of $\Omega_{N}$ ($\Omega_{N-1}$ if the last iteration of Proposition \ref{prop:induction_step} stops at $\Omega_{N-1} \in \P_{stop}^{N-2}$) is bounded below by
\begin{align*}
    &(\text{width of } \Omega_0)H(\Omega_0)^{\alpha} H(\Omega_1)^{\alpha}H(\Omega_2)^{\alpha} ... \\
    &\geq c \sigma \cdot \sigma^{5\alpha}(\sigma^{5\alpha})^{1-\alpha/4}(\sigma^{5\alpha})^{(1-\alpha/4)^2}\cdot ... \\
    &= c \sigma \cdot \sigma^{5\alpha ( 1+ (1-\alpha/4)+(1-\alpha/4)^2+...} \ge c \sigma \cdot \sigma^{5\alpha \frac{1}{1-(1-\alpha/4)}} = c \sigma^{21},
\end{align*}
for any absolute constant $c\geq 1$.

We have shown that $\Omega_N$ is $(\phi,\sigma,R,\varepsilon)$-admissible if $\Omega_N\in \P_{iter}^{\Omega_{N-1}}$.

On the other hand, if $\Omega_{N-1} \in \mathcal{P}^{\Omega_{N-2}}_{stop}$, the width of $\Omega_{N-1}$ is $\sim R^{-1}$ by property (\ref{item:4,induction_step}) of Proposition \ref{prop:induction_step}. Hence, $R^{-1} \geq \sigma^{21}$ and $\Omega_{N}$ is $(\phi,\sigma,R,\varepsilon)$-admissible.

We now analysis the decoupling constant in \eqref{eq:dec_Omega_0}. Note that the cost to decouple in each iteration in this sequence is $C_\varepsilon H(\Omega_k)^{-\varepsilon}$, except possibly the last one, so the total cost is at most 
$$
C_\varepsilon \sigma^{-O(\varepsilon)} C_\varepsilon^N \prod_{k=0}^{N} H(\Omega_k)^{-\varepsilon} \leq C_\varepsilon^{N+2} \sigma^{-O(\varepsilon)} \left (\sigma^{5}\cdot \sigma^{5(1-\alpha/4)} \cdot ... \cdot \sigma^{5(1-\alpha/4)^N} \right)^{-\varepsilon} \lesssim_\varepsilon \sigma^{-O_d(\varepsilon)} 
$$
if we pick 
\begin{equation}\label{eq:temp5.14'}
    \log K \geq c_\alpha \frac{\log C_\varepsilon}{\varepsilon} \implies \log C_\varepsilon \cdot \frac{c_\alpha\log \sigma^{-1}}{\log K} \leq \varepsilon \log \sigma^{-1} \implies C_\varepsilon^{N} \leq \sigma^{-\varepsilon}.
\end{equation}

Here, we have used the fact that $N \leq c_\alpha \log(\sigma^{-1})/\log K$. Since $\varepsilon>0$ is arbitrary, we have shown the decoupling inequality \eqref{eq:dec_Omega_0}.

It remains to prove that $50\P_\sigma^{\Omega_0}$ has $O_\varepsilon(\sigma^{-\varepsilon})$-bounded overlaps in the sense of \eqref{eq:ep_bdd_overlap}. We recall from \eqref{eq:HLowerbd} that $H(\Omega_{k}) > H(\Omega_{k-1})^{1-\alpha/4}$ whenever $k\leq N$. Hence, $H(\Omega_k) < H(\Omega_{N-1})^{(1-\alpha/4)^{k-N+1}} < K^{-(1-\alpha/4)^{k-N+1}}$. Therefore, for absolute constant $C$ in property (\ref{item:5,induction_step}),
\begin{align*}
    \sum_{k=1}^{N-1} \log (1+CH(\Omega_k)^{\alpha/d}) &\leq \sum_{k=1}^{N-1} CH(\Omega_k)^{\alpha/d}\\
    &\leq C(... + K^{-(1-\alpha/4)^{-2}\alpha/d} + K^{-(1-\alpha/4)^{-1}\alpha/d} + K^{-\alpha/d}) \\
    &\leq \log 2,
\end{align*}
by choosing $K$ large enough. Thus, we have
\begin{equation}\label{eq:temp5.15'}
    \prod(1+CH(\Omega_k)^{\alpha/d}) \leq 2.
\end{equation}

By property (\ref{item:1,induction_step}) of Proposition \ref{prop:induction_step}, the overlaps among $100 \Omega_{k+1} \in \P^{\Omega_k}$ is $O_d(1)$. Moreover, $c \Omega_{k+1}$ is contained inside $c (1+CH(\Omega_k)^{\alpha/d})\Omega_{k}$ by property (\ref{item:5,induction_step}) of Proposition \ref{prop:induction_step} . Thus,
\begin{equation}\label{eq:temp5.16}
    \sum_{\Omega_{k+1}\in \P^{\Omega_k}} 1_{c\Omega_{k+1}} \leq C_d 1_{c (1+CH(\Omega_k)^{\alpha/d})\Omega_{k}}
\end{equation}
for each $c\leq 100$. As a result of \eqref{eq:temp5.15'} and \eqref{eq:temp5.16}, we have
\begin{equation}
    \sum_{\Omega_\bullet\in \P_\sigma^{\Omega_0}} 1_{50\Omega_\bullet} \leq C_d^{N+1}\sum_{\Omega_0}1_{100\Omega_0} \leq C_d^{N+2}\leq \sigma^{-\varepsilon},
\end{equation}
where $N$ is the maximum number of iterations among all branches. We have shown the desired overlapping bound.

We have finished the proof of Theorem \ref{thm:dec_main}.

\section{Proof of Corollary \ref{cor:dec_smooth}}\label{sec:cor:dec_smooth}

In this section, we prove Corollary \ref{cor:dec_smooth}. The proof is similar to Section 2.3 of \cite{LY2021.2}.

We first tile $[-1,1]^2$ by squares $Q$ of width $\sigma^{\varepsilon}$. Let $\phi^{Q,d}$ be the Taylor polynomial of $\phi$ of degree $d$ centered at $c(Q)$, the center of $Q$. Then for any $\xi \in Q$ and any multi-index $|\beta|\leq d$, we have
$$
|\partial^\beta\phi(\xi) - \partial^\beta\phi^{Q,d}(\xi)| \lesssim_{d,\beta} \|\phi\|_{C^{d+1}([-1,1]^2)}(\sigma^{\varepsilon})^{d+1-|\beta|}.
$$
Thus, by choosing $d = 65/\varepsilon+2$, we have
\begin{equation}\label{eq:approx}
    |\partial^\beta\phi(\xi) - \partial^\beta\phi^{Q,d}(\xi)| \lesssim_{\varepsilon , \phi} \sigma^{65} \quad \forall \xi \in 10Q, |\beta|\leq 3,
\end{equation}
where the implicit constant is independent of $Q$. As a result, $Q \cap \{|\det D^2\phi|\sim \sigma\} \subset Q \cap \{|\det D^2\phi^{Q,d}|\sim \sigma\}$, with a different but positive implicit constants. 

Before we proceed to approximate $\phi$ by $\phi^{Q,d}$ and decouple $\phi^{Q,d}$ using Theorem \ref{thm:dec_main} in each $Q$, we establish the following lemma concerning the decoupled collection of sets $\Omega$ resultant from Theorem \ref{thm:dec_main}.

\begin{lem}\label{lem:error}
    Let $Q,\phi,\phi^{Q,d}$ be as above and $\Omega \subset Q$. Suppose that $\Omega$ is $(\phi^{Q,d},\sigma,r,\varepsilon)$-admissible for some $r$ satisfying item (\ref{item:2:def_curve}) in Definition \ref{def:admissible}. Then, $\Omega$ is $(\phi,\sigma,r',\varepsilon)$-admissible for any $r'\geq 1$ satisfying item (\ref{item:2:def_curve}) in Definition \ref{def:admissible}.
\end{lem}

\begin{proof}[Proof of Lemma \ref{lem:error}] Write $\phi^E = \phi^{Q,d} - \phi$. Define $l = \sup_{\xi \in \Omega} |\xi-c(\Omega)|$. We have 
$$
\sup_{[-1,1]^2} \|\phi^E_{T_\Omega}\| = \sup_{\xi \in \Omega} | \phi^E(\xi) - \phi^E(c(\Omega)) - \nabla \phi^E(c(\Omega)) \cdot (\xi-c(\Omega))| \lesssim \sigma^{65} l^2,
$$
where the last inequality follows from \eqref{eq:approx}.

We claim that 
\begin{equation}\label{eq:approx:sizephiQd}
    \|\phi^{Q,d}_{T_\Omega}\| \sim_\varepsilon \sigma^{1/2}|\Omega|.
\end{equation}
Assuming the claim for now. The width of $\Omega$ is $\gtrsim \sigma^{21}$ by item (\ref{item:2:def_curve}) in Definition \ref{def:admissible}. By \eqref{eq:approx}, we have $|\det D^2 \phi| \sim |\det D^2 \phi^{Q,d}| \sim \sigma$ on $10\Omega$. By \eqref{eq:approx:sizephiQd}, we have
$$
\|\phi^{Q,d}_{T_\Omega} - \phi_{T_\Omega}\| \leq \sigma^{65}l^2 \lesssim \sigma^{43} |\Omega| l \ll \sigma^{1/2}|\Omega| \sim \|\phi^{Q,d}_{T_\Omega}\|.
$$
Therefore, $\|\phi_{T_\Omega}\| \sim \|\phi^{Q,d}_{T_\Omega}\|$ and $|\det D^2 \bar \phi_{T_\Omega}| \sim_\varepsilon |\det D^2 \bar \phi^{Q,d}_{T_\Omega}| \sim_\varepsilon 1$.

By a similar argument, we have for each $|\beta| =2,3$
$$
\sup_{[-1,1]^2}|\partial^\beta \phi^{E}_{T_\Omega}| \lesssim \sigma^{65}l^3 \ll \sigma^{3/2}|\Omega|^3 \sim_\varepsilon \|\phi_{T_\Omega}^{Q,d}\|^3.
$$
As a result, $\sup_{[-1,1]^2}|\partial^\beta \phi^{Q,d}_{T_\Omega}| \lesssim_\varepsilon 1 $ implies $\sup_{[-1,1]^2}|\partial^\beta \phi_{T_\Omega}| \lesssim_\varepsilon 1$. We have shown that $\Omega$ is $(\phi,\sigma,R,\varepsilon)$-admissible satisfying item (\ref{item:2:def_curve}) in Definition \ref{def:admissible} assuming the claim \eqref{eq:approx:sizephiQd}.

To see the claim \eqref{eq:approx:sizephiQd}, we compute $|\det D^2 \phi^{Q,d}_{T_\Omega}|$ in two different ways. First, $|\det D^2 \phi^{Q,d}| \sim \sigma$ and hence
\begin{equation}\label{eq:temp:6.3}
    |\det D^2 \phi^{Q,d}_{T_\Omega}| \sim \sigma|\Omega|^2.
\end{equation}
On the other hand, $|\det D^2 \bar \phi^{Q,d}_{T_\Omega}| \sim_\varepsilon 1$ implies that
\begin{equation}\label{eq:temp:6.4}
    |\det D^2 \phi^{Q,d}_{T_\Omega}| \sim_\varepsilon \|\phi^{Q,d}_{T_\Omega}\|^2.
\end{equation}
The claim \eqref{eq:approx:sizephiQd} is obtained by combining \eqref{eq:temp:6.3} and \eqref{eq:temp:6.4}.
\end{proof}

Let $F$ be a function Fourier supported on $\mathcal N_{R^{-1}}^\phi(Q \cap \{|\det D^2\phi|\sim \sigma\})$, which we will approximate depending on the size of $\sigma$.

\textbf{Case 1.} $\sigma^{21} \leq R^{-1}$. In this case,
\begin{equation}\label{eq:approxcase1}
    |\partial^\beta\phi(\xi) - \partial^\beta\phi^{Q,d}(\xi)| \ll_{\varepsilon , \phi} \min\{ \sigma , R^{-1}\} \quad \forall \xi \in Q, |\beta|\leq 3,
\end{equation}
by \eqref{eq:approx}. In particular, we can approximate
$$
\mathcal N_{R^{-1}}^\phi(Q \cap \{|\det D^2\phi|\sim \sigma\}) \subset \mathcal N_{2 R^{-1}}^{\phi^{Q,d}}(Q \cap \{|\det D^2\phi^{Q,d}|\sim \sigma\}).
$$

We apply Theorem \ref{thm:dec_main} to each polynomial $\phi^{Q,d}$ and get families $\mathcal{P}_{\sigma'}=\mathcal{P}_{\sigma'}(R,\phi^{Q,d},\varepsilon)$ for dyadic number $\sigma' \in [R^{-2},1]$. However, if $\sigma' \in [R^{-2},\sigma/2)$, $\Omega \in \mathcal{P}_{\sigma'}$ has no intersection with $Q \cap \{|\det D^2\phi|\sim \sigma\}$. We see that $\mathcal{P}_{\ge \sigma} := \cup_{\sigma'\geq \sigma/2}\mathcal{P}_{\sigma'}$, where the union is over dyadic numbers $\sigma'$, satisfies items (\ref{item:1:cor:dec_smooth}) and (\ref{item:2:cor:dec_smooth}). 

To see item (\ref{item:3:cor:dec_smooth}), we need to show that $(\phi^{Q,d},\sigma',R,\varepsilon)$-admissibility implies $(\phi,\sigma',R,\varepsilon)$-admissibility. Let $\Omega$ be $(\phi^{Q,d},\sigma',R,\varepsilon)$-admissible. By Lemma \ref{lem:error}, it suffices to only consider $\Omega$ satisfying property (\ref{item:1:def_flat}) in Definition \ref{def:admissible}. By \eqref{eq:approxcase1}, 
$$|\det D^2 \phi - \det D^2 \phi^{Q,d}|\ll_{\varepsilon,\phi} \sigma.
$$ 
Since $\sigma \leq 2\sigma'$, $|\det D^2 \phi| \lesssim \sigma'$ on $10 \Omega$. Therefore, $\Omega$ is $(\phi,\sigma',R,\varepsilon)$-admissible. We have seen item (\ref{item:3:cor:dec_smooth}).

Note that there are $\sim \log \sigma^{-1} \lesssim_\varepsilon \sigma^{-\varepsilon}$ many dyadic numbers $\sigma' \in [\sigma/2,1]$. Since we allow $\sigma^{-\varepsilon}$ loss, (\ref{item:4:cor:dec_smooth}) follows from triangle and H\"{o}lder's inequalities. 

\textbf{Case 2.} $R^{-1} < \sigma^{21}$. In this case, we approximate
$$
\mathcal N_{R^{-1}}^\phi(Q \cap \{|\det D^2\phi|\sim \sigma\}) \subset \mathcal N_{\sigma^{65}}^{\phi^{Q,d}}(Q \cap \{|\det D^2\phi^{Q,d}|\sim \sigma\}).
$$

Similarly, we apply Theorem \ref{thm:dec_main} to each polynomial $\phi^{Q,d}$ and get families $\mathcal{P}_{\sigma'}=\mathcal{P}_{\sigma'}(\sigma^{-65},\phi^{Q,d},\varepsilon)$ for dyadic number $\sigma' \in [\sigma ,1]$. Let $\Omega$ be $(\phi^{Q,d},\sigma',\sigma^{-65},\varepsilon)$-admissible. Since $(\sigma')^{21} \lesssim \sigma^{65}$ cannot hold, $\Omega$ satisfies property (\ref{item:2:def_curve}) in Definition \ref{def:admissible} for $\phi^{Q,d}$. By Lemma \ref{lem:error}, $\Omega$ is $(\phi,\sigma',R,\varepsilon)$-admissible. The proofs of items (\ref{item:1:cor:dec_smooth}), (\ref{item:2:cor:dec_smooth}) and (\ref{item:4:cor:dec_smooth}) are similar to those in Case 1.

We have completed the proof of Corollary \ref{cor:dec_smooth}.

\section{A counterexample}\label{sec:counterexample}

In this section, we prove Proposition \ref{prop:counterexample}. Instead of using $\phi$ in the statement of Proposition \ref{prop:counterexample}, we carry out the computation on a more general $\phi(\xi) = \psi(|\xi|)$, where
$$
\psi(r) = 
\begin{cases}
e^{-1/r}\sin(r^{-k}) \quad & \text{if } r > 0;\\
0 \quad & \text{otherwise},
\end{cases}
$$
for some $3\leq k \in \mathbb{N}$.

Let $p,q$ be on the scaling line: $2-\frac{2}{p}=\frac{1}{q}$ and $1\leq q < \infty$. Suppose that there exists constant $C$ such that the following holds for all $f\in \mathcal{S}(\R^3)$:
\begin{equation}\label{eq:counterexample_sec7}
    \|\hat f\|_{L^q(d\mu)} \leq C \|f\|_{L^p}.
\end{equation}

By duality, see for instance Proposition 1.27 of \cite{Demeter2020}, \eqref{eq:counterexample_sec7} is equivalent to the following extension formulation
\begin{equation}\label{eq:duality_7}
    \|\widehat{fd\mu}\|_{L^{p'}}\leq C \|f\|_{L^{q'}(d\mu)}, \quad \forall f \in L^{q'}(d\mu).
\end{equation}
where $p',q'$ are the dual exponents of $p,q$ satisfying $\frac{2}{p'}=\frac{1}{q}$.

Let $f= \chi_{B(0,\frac{1}{n})}$ for some $n\gg 1$. 

Note that if $|(x_1,x_2)|\leq \frac{n}{10}$ and $|x_3| \leq e^n/10$, we have
$$
|\widehat{f d \mu}(x)| \gtrsim \int_{B(0,\frac{1}{n})} |\det D^2 \phi(\xi)|^{
\frac{1}{4}} d\xi.
$$
Thus,
$$
\|\widehat{fd\mu}\|_{L^{p'}} \gtrsim (n e^n)^{1/p'} \int_{B(0,\frac{1}{n})} |\det D^2 \phi(\xi)|^{
\frac{1}{4}} d\xi.
$$
On the other hand,
$$
\|f\|_{L^{q'}(d\mu)} = \left (\int_{B(0,\frac{1}{n})} |\det D^2 \phi(\xi)|^{
\frac{1}{4}} d\xi \right)^{1/q'}.
$$
Rearranging, \eqref{eq:duality_7} can be rewritten as
\begin{equation} \label{eq:7.3}
     \left (\int_{B(0,\frac{1}{n})} |\det D^2 \phi(\xi)|^{
\frac{1}{4}} d\xi \right)^{1/q} \leq C(n^{-1} e^{-n})^{1/p'}.
\end{equation}

We now compute the integral. Using $\det D^2 \phi(\xi) = \frac{\psi'(|\xi|)\psi''(|\xi|)}{|\xi|}$, we have $\det D^2 \phi(\xi)$ is a finite sum of the following forms
$$
c_1e^{-2/r}r^{-c_2} \sin{r^{-k}} \cos{r^{-k}}, c_1e^{-2/r}r^{-c_2} \sin{r^{-k}} \sin{r^{-k}}, c_1e^{-2/r}r^{-c_2} \cos{r^{-k}} \cos{r^{-k}}
$$
when $|\xi|=r$. The term with the largest $c_2$ dominates when $r$ is small. It is given by
$$
c_ke^{-2/r}r^{-(3k+4)} \sin{r^{-k}} \cos{r^{-k}},
$$
for some $c_k \neq 0$.

Therefore, the integral on the left hand side \eqref{eq:7.3} is bounded below by the following.
\begin{align*}
    \int_{B(0,\frac{1}{n})} |\det D^2 \phi(\xi)|^{
    \frac{1}{4}} d\xi &\gtrsim \int_{0}^{2\pi} \int_{0}^{\frac{1}{n}
    } e^{-1/2r}r^{-(3k+4)/4} |\sin({2r^{-k}})|^{1/4} r dr  d\theta\\
    &\sim \int_n^{\infty}e^{-t/2}t^{3k/4-2}|\sin({2t^{k}})|^{1/4} dt\\
    &\geq e^{-n/2}n^{3k/4-2}.
\end{align*}

Now, \eqref{eq:7.3} and above implies that
$$
(e^{-n/2}n^{3k/4-2})^{1/q} \leq C' (n^{-1}e^{-n})^{1/p'}
$$
for some finite $C'$.

Using $1/q = 2/p'$, we have
$$
n^{3k/2-4} \leq (C')^{p'} n^{-1}.
$$
This is impossible for large enough $n$, unless $p'=\infty$. The scaling condition then implies $q=\infty$, contradicting the assumption. We have finished proving Proposition \ref{prop:counterexample}.

\section{Appendix}
In this section, we prove Theorem \ref{thm:main} by Theorem \ref{thm:main_eqv} for completeness. The proof follows from Proposition 1.27 of \cite{Demeter2020} with adaption to the new measures. We only prove the implication from \eqref{eq:affine-rest_damped_eqv} to \eqref{eq:affine-rest_damped}. The other three cases are similar.

First, by duality, see Lemma 7.3 of \cite{Wo03book}, \eqref{eq:affine-rest_damped} is equivalent to the following extension formulation
\begin{equation}\label{eq:duality}
    \|\widehat{fd\mu_{\varepsilon}}\|_{L^4} \lesssim_{S,\varepsilon} \|f\|_{L^2(d\mu_{\varepsilon})}.
\end{equation}

Let $\psi_R\in \mathcal{S}(\R)$ be such that
\begin{enumerate}
    \item $\check \psi_1$ is non-negative and supported on $[-1,1]$;
    \item $1_{[-1,1]} \leq \psi_1$.
    \item $\psi_R = \psi_1(R^{-1}\cdot)$.
\end{enumerate}
Let $F(x) = \widehat{fd\mu_{\varepsilon}}(x) \psi_R(x_3)$. Then we have
$$
\check{F}(\xi,\eta) = f(\xi,\phi(\xi))|\det D^2 \phi(\xi)|^{1/4+\varepsilon} \check{\psi}_R(\eta-\phi(\xi)).
$$
We also define $dm_\varepsilon$ such that
$$
dM_\varepsilon(\xi,\eta) = |\det D^2 \phi(\xi)|^{-1/4-\varepsilon} d\xi d\eta = dm_\varepsilon(\xi,\phi(\xi))d\eta
$$

Now, we have
\begin{align*}
    \|\widehat{fd\mu_{\varepsilon}}\|_{L^4(B_R)} &\leq \|F\|_{L^4(B_R)} \\
    &\lesssim_{S,\varepsilon} R^{-1/2}\|\check{F}\|_{L^2(dM_{\varepsilon})} \\
    &\lesssim_{\psi} R^{-1/2} \|f|\det D^2 \phi|^{1/4+\varepsilon}\|_{L^2(dm_\varepsilon)} R^{1/2}\\ 
    &= \|f\|_{L^2(d\mu_\varepsilon)},
\end{align*}
where the third-to-last line follows from \eqref{eq:affine-rest_damped_eqv} in Theorem \ref{thm:main_eqv}, and the second-to-last line follows from Fubini's theorem. By letting $R\to \infty$, we have \eqref{eq:duality}. 

We have finished proving Theorem \ref{thm:main}.
\bibliographystyle{empty}
\bibliography{reference}

\end{document}